\documentclass[12pt,a4paper,leqno]{article}

\usepackage[utf8]{inputenc}
\usepackage[T1]{fontenc}
\usepackage[english]{babel}
\usepackage{amsthm}
\usepackage{amscd}
\usepackage{amsfonts}         
\usepackage{amsmath}
\usepackage{amssymb}
\usepackage{hyperref}
\usepackage{enumerate}
\usepackage{tikz}
\usepackage{tikz-cd}
\usepackage{mathrsfs}  

\newcommand{\Ab}{\mathbb{A}}

\newcommand{\Gb}{\mathbb{G}}

\newcommand{\Lb}{\mathbb{L}}

\newcommand{\Pb}{\mathbb{P}}

\newcommand{\Zb}{\mathbb{Z}}

\newcommand{\Cc}{\mathcal{C}}

\newcommand{\Ec}{\mathcal{E}}
\newcommand{\Fc}{\mathcal{F}}

\newcommand{\Nc}{\mathcal{N}}
\newcommand{\Oc}{\mathcal{O}}
\newcommand{\Pc}{\mathcal{P}}

\newcommand{\Vc}{\mathcal{V}}

\newcommand{\Ls}{\mathscr{L}}

\newcommand{\Ifr}{\mathfrak{I}}

\newcommand{\mfr}{\mathfrak{m}}

\newcommand{\pfr}{\mathfrak{p}}
\newcommand{\qfr}{\mathfrak{q}}

\newcommand{\PCob}{{\underline\Omega}}

\newcommand{\coim}{\mathrm{\coim}}

\newcommand{\Spec}{\mathrm{Spec}}

\newcommand{\bl}{\mathrm{Bl}}

\newcommand{\wtil}{\widetilde}

\newcommand{\abs}[1]{\lvert #1 \rvert}

\newcommand{\rank}{\mathrm{rank}}

\newcommand{\hook}{\hookrightarrow}
\newcommand{\cl}{\mathrm{cl}}
\newcommand{\modmod}{/\!\!/}

\newtheorem{theo}{Tplottin ubuntuheorem}[section]
\theoremstyle{plain}
\newtheorem{thm}[theo]{Theorem}
\newtheorem{lem}[theo]{Lemma}
\newtheorem{prop}[theo]{Proposition}
\newtheorem{cor}[theo]{Corollary}
\newtheorem{con}[theo]{Conjecture}
\newtheorem*{thm*}{Theorem}
\newtheorem*{lem*}{Lemma}
\newtheorem*{prop*}{Proposition}
\newtheorem*{cor*}{Corollary}

\theoremstyle{definition}
\newtheorem{defn}[theo]{Definition}
\newtheorem{ex}[theo]{Example}
\newtheorem{cons}[theo]{Construction}
\newtheorem{rem}[theo]{Remark}

\title{Algebraic Spivak's theorem and applications}
\author{Toni Annala}

\newcommand{\Addresses}{{
  \bigskip
  \footnotesize

  Toni Annala, \textsc{Department of Mathematics, University of British Columbia,
    Vancouver, BC V6T1Z2 Canada}\par\nopagebreak
  \textit{E-mail address:} \texttt{tannala@math.ubc.ca}

}}

\date{}
\begin{document}

\maketitle

\begin{abstract}
We prove an analogue of Lowrey--Schürg's algebraic Spivak's theorem when working over a base ring $A$ that is either a field or a nice enough discrete valuation ring, and after inverting the residual characteristic exponent $e$ in the coefficients. By this result algebraic bordism groups of quasi-projective derived $A$-schemes can be generated by classical cycles, leading to vanishing results for low degree $e$-inverted bordism classes, as well as to the classification of quasi-smooth projective $A$-schemes of low virtual dimension up to $e$-inverted cobordism. As another application, we prove that $e$-inverted bordism classes can be extended from an open subset, leading to the proof of homotopy invariance of $e$-inverted bordism groups for quasi-projective derived $A$-schemes.
\end{abstract}

\tableofcontents

\section{Introduction}

Algebraic cobordism $\Omega^*$ is supposed to be the universal oriented cohomology theory (and more generally an oriented bivariant theory) in algebraic geometry. It should contain more refined information than intersection theory and algebraic $K$-theory, and in fact the latter two theories should be recoverable from $\Omega^*$ in a formal fashion. Of course, such a theory should not be easily computable. Rather, its role is to serve as the canvas for the most general possible (oriented) cohomology computations in algebraic geometry, and it should clarify the geometric essence of formulas of which other theories (such as intersection theory and algebraic $K$-theory) only see shadows of (compare this to the role of Grothendieck ring of varieties in the study of Euler characteristics).

However, the existence of algebraic cobordism is largely hypothetical: until rather recently, the only known geometric model for such a theory was the algebraic bordism theory $\Omega_*$ of Levine and Morel constructed in \cite{Levine:2007} for algebraic schemes over a field of characteristic 0. Note that $\Omega_*$ is a homology theory rather than a cohomology theory, and therefore the groups $\Omega_*(X)$ recover the algebraic cobordism rings only for smooth varieties (this is a version of Poincaré duality). Even this limited model has had considerable successes over the years: Levine and Pandharipande proved the degree 0 Donaldson--Thomas conjectures in \cite{levine-pandharipande} using the computation of algebraic cobordism of a field of characteristic 0, Hudson--Ikeda--Matsumura--Naruse study certain $K$-theoretical characteristic classes in \cite{hudson-ikeda-matsumura-naruse} using algebraic cobordism (see also the follow-up paper \cite{hudson-matsumura} by Hudson and Matsumura generalizing some of the results to algebraic cobordism), and Sechin and Sechin--Semenov have found applications of algebraic Morava $K$-theories (defined using algebraic cobordism) to approximating torsion in Chow groups in \cite{sechin-chern} and to arithmetic questions about algebraic groups in \cite{sechin-semenov}.

Combining on the innovative approach of Lowrey and Schürg in \cite{lowrey-schurg} of using derived algebraic geometry in the study of algebraic cobordism and the work by Yokura on universal bivariant theories in \cite{yokura09}, the author was able to extend in \cite{annala-cob} algebraic bordism of Levine and Morel to a full fledged bivariant theory for quasi-projective derived schemes over a field of characteristic 0, giving for the first time geometric models for algebraic cobordism (and Chow cohomology) rings $\Omega^*(X)$ of singular varieties. Note that the rings $\Omega^*(X)$ are not $\Ab^1$-invariant, and hence cannot be represented by motivic spectra in the stable $\Ab^1$-homotopy theory. Since then, it was realized that the methods of derived algebraic geometry have potential for defining algebraic cobordism very generally, and starting from \cite{annala-yokura} and \cite{annala-chern} a program was initiated with the goal of arriving at the general definition of (bivariant) algebraic cobordism, and proving all the expected properties. After solving several technical issues in \cite{annala-pre-and-cob} and \cite{annala-base-ind-cob}, we have obtained a bivariant theory $\Omega'^\bullet$ on the $\infty$-category of finite dimensional Noetherian derived schemes admitting an ample family of line bundles, where the groups $\Omega'^\bullet(X \to Y)$ admit a natural grading whenever $X \to Y$ is of finite type.

Note that as special cases of the bivariant groups we obtain models of algebraic cobordism rings
$$\Omega'^*(X) := \Omega'^*(X \to X)$$
and algebraic bordism groups
$$\Omega'_\bullet(X) := \Omega'_\bullet\big(X \to \Spec(\Zb) \big)$$ 
giving rise to a cohomology and a Borel--Moore homology theory respectively. Moreover, these invariants are known to have the following properties 
\begin{enumerate}
\item if $X$ is quasi-projective over a field of characteristic 0, then $\Omega_\bullet(X)$ recovers the algebraic bordism group of Levine and Morel;

\item if $X$ is regular, then there is a natural Poincaré-duality isomorphism
$$- \bullet 1_{X/\Zb}: \Omega'^*(X) \xrightarrow{\cong} \Omega'_\bullet(X);$$

\item the theory $\Omega'^\bullet$ satisfies the bivariant analogue of projective bundle formula;

\item given a vector bundle $E$ on $X$, there exists Chern classes 
$$c_i(E) \in \Omega'^i(X)$$
satisfying the expected properties;

\item $K^0(X)$ can be recovered from $\Omega'^*(X)$ analogously to Conner--Floyd theorem.
\end{enumerate}
The purpose of this paper is to continue the study of the properties of our model $\Omega'^\bullet$ of bivariant algebraic cobordism. In contrast with the earlier work, our emphasis will be on studying the homology theory $\Omega'_\bullet$, and the results will hold only for quasi-projective derived schemes over fields and some discrete valuation rings. 

One of the main motivations for this work is the following conjectural computation, whose only known case is when $A$ is a field of characteristic 0 (this is due to \cite{Levine:2007}).
\begin{con}\label{CobordismOfLocalRingsConj}
Let $A$ be a local Noetherian ring. Then the natural map
$$\Lb^* \to \Omega'^* \big( \Spec(A) \big)$$
is an isomorphism, where $\Lb^*$ is the Lazard ring (with cohomological grading).
\end{con}
\noindent Proving this result would be very interesting even in the case where $A=k=\bar k$ is an algebraically closed field of positive characteristic, as the proof method would have to be different from the one used in characteristic 0 (which is based on resolution of singularities and weak factorization). Even though we manage to make only modest progress towards this conjecture, the results we obtain in the progress of doing so (notably Algebraic Spivak's theorem) have major ramifications for the structure of algebraic bordism groups after inverting residual characteristic, as we shall see below.

\subsection{Summary of results}

Recall that the bivariant theory $\Omega'^\bullet$ coincides with a simpler theory $\Omega^\bullet$ (also constructed in \cite{annala-base-ind-cob}) for derived schemes admitting an ample line bundle (instead of just an ample family). The main result of this article is the following analogue of the Algebraic Spivak's theorem of Lowrey and Schürg from \cite{lowrey-schurg} (see also the original result of Spivak in derived differential geometry, \cite{spivak2010} Theorem 3.12). 

\begin{thm*}[Theorem \ref{GeneralSpivakThm}]
Let $A$ be a field or an excellent Henselian discrete valuation ring with a perfect residue field, and let $e$ be the residual characteristic exponent of $A$. Then the algebraic cobordism ring $\Omega^*\big(\Spec(A)\big)[e^{-1}]$ is generated as a $\Zb[e^{-1}]$-algebra by regular projective $A$-schemes. Moreover, for all quasi-projective derived $A$-schemes $X$, $\Omega_\bullet(X)[e^{-1}]$ is generated as an $\Omega^*\big(\Spec(A)\big)[e^{-1}]$-module by classes of regular schemes mapping projectively to $X$.
\end{thm*}

\noindent In particular, the $e$-inverted algebraic bordism groups are generated by derived fibre products over $A$ of regular $A$-schemes. It is possible to use this fact to prove that classical cycles generate these groups as $\Zb[e^{-1}]$-modules, leading to the following corollaries.

\begin{cor*}[Corollary \ref{StrongerSpivakCor}]
Suppose $A$ is as in Theorem \ref{GeneralSpivakThm}. Then, for all quasi-projective derived $A$-schemes $X$, $\Omega_\bullet(X)[e^{-1}]$ is generated as a $\Zb[e^{-1}]$-module by cycles of form
$$[V \to X]$$
with $V$ a classical complete intersection scheme and the structure morphism $V \to \Spec(A)$ is either flat or factors through the unique closed point of $\Spec(A)$.
\end{cor*}

\begin{cor*}[Corollary \ref{PerfectSpivakCor}]
If $k$ is a perfect field, then the $e$-inverted bordism groups $\Omega_\bullet(X)[e^{-1}]$, where $X$ is a quasi-projective derived $k$-scheme, are generated as $\Zb[e^{-1}]$-modules by classes of smooth $k$-varieties mapping projectively to $X$.
\end{cor*}
\noindent Note that these results are not strictly speaking generalizations of the results of Lowrey--Schürg: since we do not have a classically defined (geometric) bordism theory against which to compare $\Omega_\bullet$, the best statement we can hope for is this kind of a result stating that classical cycles are enough to generate the bordism groups. Notice also how these results bring us closer to proving Conjecture \ref{CobordismOfLocalRingsConj} after inverting the residual characteristic: instead of classifying all quasi-smooth and projective derived $A$-schemes up to cobordism, we only have to classify the lci projective $A$-schemes, or even just smooth projective $k$-varieties if $A = k$ is a perfect field, up to cobordism, and this seems like a much easier problem (and can be done in low dimensions, as we shall see below).

Most of the article is dedicated to proving the above results. The basic idea is to compute the fundamental class of a quasi-projective derived scheme $X$ by resolving the singularities of the deformation to the normal cone of the truncated inclusion $X_\cl \hook U$, where $U$ is an open subscheme of $\Pb^n_A$, and then using the special presentations of Chern classes given by Proposition \ref{RegLineChernLem}. Carrying out  this strategy is made difficult by the facts that desingularization by alterations, unlike Hironaka's resolution in characteristic 0, is not a precision tool as it may drastically change the geometry outside the singular locus, and the technique does not provide birational resolutions. To solve the first problem, we need to be able to approximate classes of generically finite morphisms to projective space, which is achieved by Theorem \ref{SecondRefinedPBFThm}, and to solve the second problem we need to invert the residual characteristic. Note that it is not easy to compute the classes of generically finite morphisms is general without localization exact sequences, which is why we only do the computation when the target is a projective space. Even this case is not simple, and Theorem \ref{SecondRefinedPBFThm} is the main reason we have to restrict our attention to the case where the base ring $A$ is an excellent Henselian discrete valuation ring having a perfect residue field instead of a more general excellent discrete valuation ring. In order to generalize further to, say, the case where $A$ an excellent regular local ring of Krull-dimension $\leq 3$, one would need to generalize also the Bertini-regularity theorems from \cite{ghosh-krishna} to hold in this generality.


Theorem \ref{GeneralSpivakThm}, Corollary \ref{StrongerSpivakCor} and especially Corollary \ref{PerfectSpivakCor} considerably simplify the study of algebraic bordism whenever they apply, as derived schemes are much harder to study than classical schemes, let alone smooth varieties over a field. As the first immediate corollary, we obtain the following vanishing result.

\begin{cor*}[Corollary \ref{VanishCor}]
Let $X$ be a quasi-projective $A$-variety, with $A$ as in Theorem \ref{GeneralSpivakThm}. Then the groups 
$$\Omega^A_i(X) := \Omega^{-i}\big(X \to \Spec(A)\big)$$
vanish for $i < -1$. If $A = k$ is a field, then $\Omega^k_i(X)$ vanish for $i < 0$.
\end{cor*}

Note that the above result is far from obvious without the algebraic Spivak's theorem, as there exists an abundance of derived schemes having negative virtual dimension. One can also imagine such vanishing results being very useful, as they allow proving results by induction on degree, rather than by some ad hoc induction scheme. Another easy application is the computation of the cobordism rings in low degrees.

\begin{cor*}[Corollary \ref{CobRingCor}]
Let $A$ be as in Theorem \ref{GeneralSpivakThm}. Then 
\begin{enumerate}
\item $\Omega^i\big( \Spec(A) \big)[e^{-1}]$ vanishes for $i>0$, and the natural map 
$$\Zb[e^{-1}] \cong \Lb^0[e^{-1}] \to \Omega^0\big( \Spec(k) \big)[e^{-1}]$$
is an isomorphism; in other words $\Omega^0\big( \Spec(A) \big)$ is the free $\Zb[e^{-1}]$-module generated by $1_{A}$;

\item if $A = k$ is a field, then the natural map 
$$\Lb^*[e^{-1}] \to \Omega^*\big( \Spec(k) \big)[e^{-1}]$$
is an isomorphism in degrees $0, -1$ and $-2$; in other words $\Omega^1\big( \Spec(k) \big)$ is the free $\Zb[e^{-1}]$-module generated by $[\Pb^1_k]$ and $\Omega^2 \big( \Spec(k) \big)$ is the free $\Zb[e^{-1}]$-module generated by $[\Pb^1_k]^2$ and $[\Sigma_{1,k}]$, where $\Sigma_{1,k}$ is the Hirzebruch surface of degree 1 over $k$.
\end{enumerate}
\end{cor*}

\noindent We are hopeful that more progress towards the computation of algebraic cobordism of fields can be made in the future.

Another application, which is no longer an immediate corollary of the Algebraic Spivak's theorem, is the following extension result.
\begin{thm*}[Theorem \ref{ExtensionThm}]
Let $A$ be a field or an excellent Henselian discrete valuation ring with a perfect residue field, and let $j: X \hook \overline X$ be an open embedding of quasi-projective derived $A$-schemes. Then the pullback morphism
$$j^!: \Omega_\bullet(\overline X)[e^{-1}] \to \Omega_\bullet(X)[e^{-1}]$$
is surjective.
\end{thm*}
\noindent There are two main obstructions for proving such a result: first of all, it is unknown whether or not one can ``compactify'' derived schemes, and second of all, even if this was possible, one would then still need to ``resolve'' the singularities of the derived scheme obtained in this way to arrive at a derived complete intersection model, which seems like an even harder problem.  The first problem is solved essentially by invoking algebraic Spivak's theorem, so the proof of the Extension theorem is dedicated to overcoming the second obstruction. Note also that this result is the right side of the conjectural localization exact sequence (known in characteristic 0 by the work of Levine and Morel), but it seems hard to say anything worthwhile about the kernel of $j^!$.

Combining the Extension theorem with the projective bundle formula, we obtain the following homotopy invariance result.

\begin{cor*}[Corollary \ref{HomotopyInvarianceCor}]
Let $A$ be as in Theorem \ref{ExtensionThm}, let $X$ be a quasi-projective derived $A$-scheme and let $p: E \to X$ be a vector bundle of rank $r$ on $X$. Then the pullback map
$$p^!: \Omega_\bullet(X)[e^{-1}] \to \Omega_\bullet(E)[e^{-1}]$$
is an isomorphism.
\end{cor*}
\noindent While this result was expected as the cobordism analogue of $\Ab^1$-invariance of Chow groups and of Grothendieck groups of coherent sheaves, the reader should compare it with the fact that the cobordism rings $\Omega^*(X)$ are usually not $\Ab^1$-invariant.

\subsection{Conventions}

All derived schemes are assumed to be Noetherian and of finite Krull dimension. Derived fibre product is denoted by $X \times^R_Z Y$, and truncation is denoted by $X_\cl$. A map of derived schemes is called \emph{projective} if it is proper and admits a relatively ample line bundle. The \emph{Krull dimension} of a derived scheme means the Krull dimension of its truncation.  A flat lci morphism is called \emph{syntomic}. An \emph{snc scheme} is a scheme that can be globally expressed as an snc divisor inside a regular scheme. We will denote
$$[r] := \{1,..,r\}.$$

\subsection*{Acknowledgements}

The author would like to thank his advisor Kalle Karu for discussions and comments. The author is supported by Vilho, Yrj\"o and Kalle V\"ais\"al\"a Foundation of the Finnish Academy of Science and Letters.

\section{Background}


In this section we recall the necessary background material.

\subsection{Quasi-smooth morphisms and derived complete intersection schemes}

In this section we review the background of various notions of derived complete intersections that we are going to use in the work. Let us recall that a closed embedding $Z \hook X$ of derived schemes is called a \emph{derived regular embedding (of virtual codimension $r$)} if it is locally on $X$ of form
$$\Spec\big(A \modmod (a_1, ..., a_r)\big) \hook \Spec(A),$$
i.e., $Z$ is locally given as the derived vanishing locus of $r$ functions on $X$. Let us begin with the following definition.

\begin{defn}
A finite type morphism $f: X \to Y$ of Noetherian derived schemes is called \emph{quasi-smooth} if the relative cotangent complex $\Lb_{X/Y}$ has Tor-dimension $\leq 1$ (it follows that $\Lb_{X/Y}$ is perfect and $f$ is of finite presentation, see e.g. \cite{annala-base-ind-cob} Proposition 2.23). If $\Lb_{X/Y}$ has constant virtual rank $d$ on $X$, then we say that $f$ is quasi-smooth of \emph{relative virtual dimension $d$}.
\end{defn}

Let us then recall the basic properties of quasi-smooth morphisms.

\begin{prop}
\begin{enumerate}
\item Quasi-smooth morphisms are stable under composition and derived base change. Relative virtual dimension is additive under composition, and is preserved in derived pullbacks.

\item A morphism of classical schemes is quasi-smooth if and only if it is lci. 

\item A closed embedding is quasi-smooth if and only if it is a derived regular embedding.
\end{enumerate}
\end{prop}
\begin{proof}
The first claim is obvious, the second is classical, and the third is \cite{khan-rydh} Proposition 2.3.8.
\end{proof}

We are also going to need the following absolute version of quasi-smoothness.

\begin{defn}
A Noetherian derived scheme $X$ is called a \emph{derived complete intersection scheme} if only finitely many of the homotopy sheaves $\pi_i(\Oc_X)$ are nontrivial, and if for all points $x \in X$ the cotangent complex $\Lb_{\kappa(x) / X}$ has Tor-dimension $\leq 2$, where $\kappa(x)$ is the residue field of $X_\cl$ at $x$.
\end{defn}

Derived complete intersection schemes admit the following alternative characterization.

\begin{prop}\label{DCIAltCharProp}
Let $X$ be a Noetherian derived scheme. Then the following are equivalent:
\begin{enumerate}
\item $X$ is a derived complete intersection scheme;
\item the cotangent complex $\Lb_{X / \Zb}$ has Tor-dimension $\leq 1$;
\item for all morphisms $X \to Y$ with $Y$ a regular scheme, the relative cotangent complex $\Lb_{X / Y}$ has Tor-dimension $\leq 1$;
\item there exists a morphism $X \to Y$ with $Y$ regular so that the relative cotangent complex $\Lb_{X / Y}$ has Tor-dimension $\leq 1$.
\end{enumerate}
\end{prop}
\begin{proof}
The equivalence of 1, 2 and 3 is \cite{annala-base-ind-cob} Proposition 2.28, but the proof shows also that they are all equivalent to 4.
\end{proof}

\begin{ex}
Following types derived schemes are derived complete intersections:
\begin{enumerate}
\item classical complete intersection schemes (in particular, regular schemes);
\item if $X$ is a derived complete intersection and $f: Y \to X$ is a morphism so that $\Lb_{Y/X}$ has Tor-dimension $\leq 1$ (e.g. if $f$ is quasi-smooth), then $Y$ is a derived complete intersection scheme.
\end{enumerate}
\end{ex}

\subsection{Derived blow ups}

One of the main technical tools we are going to need in this article is the construction of derived blow ups and derived deformation to normal cone from \cite{khan-rydh}. Let us recall the definitions and the results we are going to use:

\begin{defn}\label{DivisorOverZDef}
Let $Z \hookrightarrow X$ be a derived regular embedding. Then, for any $X$-scheme $S$, a \emph{virtual Cartier divisor} on $S$ \emph{lying over $Z$} is the datum of a commutative diagram
\begin{center}
\begin{tikzcd}
D \arrow[hook]{r}{i_D} \arrow[]{d}{g} & S \arrow[]{d}{} \\
Z \arrow[hook]{r}{} & X
\end{tikzcd}
\end{center}
such that
\begin{enumerate}
\item $i_D$ is a derived regular embedding of virtual codimension 1 (i.e., a virtual Cartier divisor);
\item the truncation is a Cartesian square;
\item the canonical morphism
$$g^* \Nc^\vee_{Z/X} \to \Nc^\vee_{D/S}$$
induces a surjection on $\pi_0$.
\end{enumerate}
\end{defn}

It is then possible to define derived blow ups via its functor of points.

\begin{defn}\label{DerivedBlowUpDef}
Let $Z \hookrightarrow X$ be a derived regular embedding. Then the \emph{derived blow up} $\bl_Z(X)$ is the $X$-scheme representing virtual Cartier divisors lying over $Z$. In other words, given an $X$-scheme $S$, the space of $X$-morphisms 
$$S \to \bl_Z(X)$$
is naturally identified with the maximal sub $\infty$-groupoid of the $\infty$-category of virtual Cartier divisors of $S$ that lie over $Z$.
\end{defn}

\begin{thm}\label{BlowUpPropertiesThm}
Let $i: Z \hookrightarrow X$ be a derived regular embedding with $X$ Noetherian. Then
\begin{enumerate}
\item the derived blow up $\bl_Z(X)$ exists as a derived scheme and is unique up to contractible space of choices;

\item the structure morphism $\pi: \bl_Z(X) \to X$ is projective, quasi-smooth, and induces an equivalence 
$$\bl_Z(X) - \Ec \to X - Z,$$
where $\Ec$ is the universal virtual Cartier divisor on $\bl_Z(X)$ lying over $Z$ (also called the \emph{exceptional divisor});

\item the derived blow up $\bl_Z(X) \to X$ is stable under derived base change;

\item the exceptional divisor $\Ec$ is naturally identified with $\Pb_Z(\Nc_{Z/X})$;

\item if $Z \stackrel i \hookrightarrow X \stackrel j \hookrightarrow Y$ is a sequence of quasi-smooth closed embeddings, then there exists a natural derived regular embedding $\tilde j: \bl_Z(X) \hookrightarrow \bl_Z(Y)$ called the \emph{strict transform};

\item given derived regular embeddings $i: Z \hook X$ and $j: Y \hook X$, the strict transforms $\tilde i$ and $\tilde j$ do not meet in $\bl_{Z \cap Y}(X)$;

\item if $Z$ and $X$ are classical schemes (so that $Z \hookrightarrow X$ is lci), there exists a natural equivalence 
$$\bl_Z(X) \simeq \bl_{Z}^\mathrm{cl}(X),$$
where the right hand side is the classical blow up.
\end{enumerate}
\end{thm}
\begin{proof}
The statements 1, 3, 4, 5 and 7 are directly from \cite{khan-rydh} Theorem 4.1.5. The second claim is otherwise from \emph{loc. cit.}, but the authors only prove that $\pi$ is proper; projectivity of $\pi$ follows from the fact that the line bundle $\Oc(-\Ec)$ is $\pi$-ample. For a proof of 6, see for example \cite{annala-cob} Lemma 4.5.
\end{proof}

Another result we are going to need is the following:

\begin{prop}\label{BlowUpVsTruncationProp}
Let $Z \hook X$ be a derived regular embedding. Then there exists a natural closed embedding $\bl^\cl_{Z_\cl}(X_\cl) \hook \bl_{Z}(X)$, and a derived Cartesian square
$$
\begin{tikzcd}
\Ec_\cl \arrow[hook]{d} \arrow[hook]{r} & \bl^\cl_{Z_\cl}(X_\cl) \arrow[hook]{d} \\
\Ec \arrow[hook]{r} & \bl_Z(X),
\end{tikzcd}
$$
where $\Ec_\cl$ is the classical exceptional divisor.
\end{prop}
\begin{proof}
Indeed, the closed embedding $i_{Z/X}$ from \cite{annala-pre-and-cob} Appendix B.4 has this property, see Lemma B.15 and Theorem B.16.
\end{proof}

\subsection{Algebraic cobordism}

In this section, we will recall algebraic cobordism for finite dimensional Noetherian derived schemes having an ample line bundle, since we will not need it in greater generality. Throughout this section we will denote by $\Cc$ the $\infty$-category of finite dimensional Noetherian derived schemes admitting an ample line bundle. 

\subsubsection{Definition of algebraic (co)bordism}

Let us start with the definition of universal precobordism.

\begin{defn}\label{AlgCobDefn}
Let $X \to Y$ be a morphism in $\Cc$. Then the universal precobordism group $\PCob^\bullet(X \to Y)$ is the group completion of the Abelian monoid on cycles
$$[V \xrightarrow{f} X],$$
where $f$ is projective, $\Lb_{V/Y}$ has Tor-dimension $\leq 1$, and where the monoid operation is given by taking disjoint union. These cycles are subjected to the \emph{derived double point relations}: given a projective morphism $W \to \Pb^1 \times X$ so that $\Lb_{W/\Pb^1 \times Y}$ has Tor-dimension $\leq 1$, and virtual Cartier divisors $D_1$ and $D_2$ on $W$ so that their sum is the fibre $W_\infty$ of $W \to \Pb^1$ over $\infty$, we have that
$$[W_0 \to X] = [D_1 \to X] + [D_2 \to X] - [\Pb_{D_1 \times^R_W D_s}(\Oc(D_1) \oplus \Oc) \to X] \in \PCob^\bullet(X),$$
where $W_0$ is the fibre of $W \to \Pb^1$ over $0$. To see how the bivariant operations are defined, see for example \cite{annala-base-ind-cob} Definition 2.38.
\end{defn}

\begin{defn}[Bivariant fundamental classes]
Given a morphism $X \to Y$ in $\Cc$ so that $\Lb_{X/Y}$ has Tor-dimension $\leq 1$, we have the \emph{bivariant fundamental class}
$$1_{X/Y} := [X \to X] \in \PCob^\bullet(X \to Y).$$
These classes give rise to a stable orientation on the bivariant theory $\PCob^\bullet$. If $X \to X$ is the identity morphism, we will often use the short hand notation $1_X := 1_{X/X}$, and if either $X$ or $Y$ is a spectrum of a ring, we will often drop ``$\Spec$'' from the notation.
\end{defn}

It was shown in \cite{annala-base-ind-cob} Theorem 3.15 that there exists a formal group law 
$$F(x,y) = \sum_{i,j \geq 0} a_{ij} x^i y^j \in \PCob^\bullet\big(\Spec(\Zb)\big)[[x,y]]$$
so that for every $X \in \Cc$ and for all line bundles $\Ls_1$ and $\Ls_2$ on $X$, the equality
$$c_1(\Ls_1 \otimes \Ls_2) = F\big(c_1(\Ls_1), c_1(\Ls_2)\big) \in \PCob^\bullet(X)$$
holds, where $c_1(\Ls)$ is the \emph{first Chern class} (also called the \emph{Euler class}) of $\Ls$, i.e., the cycle represented by $[Z_s \hook X]$ where $s$ is any global section of $\Ls$ and $Z_s$ is the derived vanishing locus of $s$. 

\begin{defn}\label{VStackDefn}
Consider the stack $[\Ab^1 / \Gb_m] \times [\Ab^1 / \Gb_m]$, which classifies the data
$$(\Ls_1, s_1, \Ls_2, s_2),$$
where $\Ls_i$ are line bundles and $s_i$ are global sections of $\Ls_i$. Let 
$$(\Ls^u_1, s^u_1, \Ls^u_2, s^u_2)$$
be the universal such data classified by the identity morphism, and let us denote by $\Vc$, $\Vc_i$ and $\Vc_{12}$ the derived vanishing loci of $s^u_1 s^u_2$, $s^u_i$ and $(s^u_1, s^u_2)$ inside $[\Ab^1 / \Gb_m] \times [\Ab^1 / \Gb_m]$. Notice that there exists a commutative square
$$
\begin{tikzcd}
\Vc_{12} \arrow[hook]{r} \arrow[hook]{d} \arrow[hook]{rd}{\iota^{12}} & \Vc_2 \arrow[hook]{d}{\iota^2} \\
\Vc_1 \arrow[hook]{r}{\iota^1} & \Vc,
\end{tikzcd}
$$
and none of the maps $\iota^1, \iota^2$ and $\iota^{12}$ is a derived regular embedding. The embeddings $\Vc_{12} \hook \Vc_i$ are derived regular.
\end{defn}

We are now ready to recall the definition of algebraic cobordism.

\begin{defn}
Let $X \to Y$ be a morphism in $\Cc$. Then the \emph{bivariant algebraic cobordism} group $\Omega^\bullet(X \to Y)$ is obtained by imposing the following \emph{decomposition relation} on $\PCob^\bullet(X \to Y)$: given a projective morphism $f: V \to X$ so that $\Lb_{V/Y}$ has Tor-dimension $\leq 1$, and a morphism $V \to \Vc$, denote by $V_i$ and $V_{12}$ the derived pullbacks $V \times_\Vc \Vc_i$ and $V \times_\Vc \Vc_{12}$ and by $f^i$ and $f^{12}$ the induced projective morphisms $V_i \to X$ and $V_{12} \to X$ respectively. If $\Lb_{V_1/Y}$ and $\Lb_{V_2/Y}$ have Tor-dimension $\leq 1$, then 
$$[V \to X] = [V_1 \to X] + [V_2 \to X] + f^{12}_* \Bigg( \sum_{i,j \geq 1} a_{ij} c_1(\Ls_1)^{i-1} \bullet c_1(\Ls_2)^{j-1} \bullet 1_{V_{12} / Y} \Bigg)$$
in $\Omega^\bullet(X \to Y)$ where $\Ls_i$ are the pullbacks of the universal line bundles $\Ls^u_i$, and where $1_{V_{12}/Y}$ is the bivariant fundamental class
$$[V_{12} \to V_{12}] \in \PCob^\bullet(V_{12} \to Y).$$
\end{defn}

Let us then recall in more detail the homology and cohomology theories associated to $\Omega^\bullet$, starting with the homology theory.

\begin{defn}[Algebraic bordism]
Given $X \in \Cc$, we define its \emph{algebraic bordism} group as 
$$\Omega_\bullet(X) := \Omega^\bullet\big(X \to \Spec(\Zb) \big).$$
Let $f: X \to Y$ be a morphism in $\Cc$. If $f$ is projective, then there exists a \emph{pushforward morphism}
$$f_*: \Omega_\bullet(X) \to \Omega_\bullet(Y)$$
given by the formula
$$[V \xrightarrow{g} X] \mapsto [V \xrightarrow{f \circ g} Y],$$
and if $\Lb_f$ has Tor-dimension $\leq 1$, there exists \emph{Gysin pullback morphism}
$$f^!: \Omega_\bullet(Y) \to \Omega_\bullet(X)$$
given by the formula
$$[W \to Y] \mapsto [W \times^R_Y X \to X].$$
Pushforwards and Gysin pullbacks are functorial in the obvious sense.
\end{defn}

The following result is an immediate consequence of Proposition \ref{DCIAltCharProp} and the bivariant formalism.

\begin{prop}\label{BordismVsYBordismProp}
Let $Y \in \Cc$ be a regular scheme. Then for all $X \to Y$ in $\Cc$ the morphism
$$- \bullet 1_{Y / \Zb}: \Omega^Y_\bullet(X) := \Omega^\bullet(X \to Y) \to \Omega_\bullet(X)$$
is an isomorphism, and these maps commute with pushforwards and Gysin pullbacks. \qed
\end{prop}

Let us then recall the cohomology theory.

\begin{defn}[Algebraic cobordism]
Given $X \in \Cc$, we define its \emph{algebraic cobordism} ring as
$$\Omega^*(X) := \Omega^\bullet(X \to X),$$
where the ring structure is given by the formula
$$[V \to X] \bullet [W \to X] = [V \times^R_X W \to X] \in \Omega^*(X).$$
Notice that $\Omega^*(X)$ is a graded ring, and $[V \to X]$ is of degree $d$ if $V \to X$ is of relative virtual dimension $-d$. Suppose then that $f: X \to Y$ is a morphism in $\Cc$. Then there exists a \emph{pullback morphism}
$$f^*: \Omega^*(Y) \to \Omega^*(X)$$
given by the formula
$$[W \to Y] \mapsto [W \times^R_Y X \to X],$$
and if $f$ is projective and quasi-smooth, there exists a \emph{Gysin pushforward morphism}
$$f_!: \Omega^*(X) \to \Omega^*(Y)$$
given by the formula
$$[V \xrightarrow{g} X] \mapsto [V \xrightarrow{f \circ g} Y].$$
Note that $f^*$ is multiplicative and preserves the grading while $f_!$ does not have to do either. Both pullbacks and Gysin pushforwards are functorial in the obvious sense.
\end{defn}

These theories have the following formal properties.

\begin{prop}[Projection formula]
The group $\Omega_\bullet(X)$ is an $\Omega^*(X)$-module, with the action given by
$$[V \to X] \bullet [W \to X] = [V \times^R_X W \to X].$$
Moreover, if $f: X \to Y$ is a projective morphism, then for all $\alpha \in \Omega^*(Y)$ and all $\beta \in \Omega_\bullet(X)$, the equality
$$f_*\big(f^*(\alpha) \bullet \beta \big) = \alpha \bullet f_*(\beta) \in \Omega_\bullet(Y)$$
holds. If $f$ is projective and quasi-smooth, then for all $\alpha \in \Omega^*(Y)$ and all $\gamma \in \Omega^*(X)$, the equality
$$f_!\big(f^*(\alpha) \bullet \gamma \big) = \alpha \bullet f_!(\gamma) \in \Omega^*(Y)$$
holds.
\end{prop}
\begin{proof}
This is an immediate consequence of bivariant formalism, but can also be easily checked on the level of cycles.
\end{proof}

\begin{prop}[Push-pull formula]
Suppose that
$$
\begin{tikzcd}
X' \arrow[]{r}{f'} \arrow[]{d}{g'} & Y' \arrow[]{d}{g} \\
X \arrow[]{r}{f} & Y  
\end{tikzcd}
$$
is a derived Cartesian square in $\Cc$. Then
\begin{enumerate}
\item if $\Lb_{f}$ has Tor-dimension $\leq 1$ and $g$ is projective,
$$f^! \circ g_* = g'_* \circ f'^! : \Omega_\bullet(Y') \to \Omega_\bullet(X);$$

\item if $g$ is projective and quasi-smooth, 
$$f^* \circ g_! = g'_! \circ f'^* : \Omega^*(Y') \to \Omega^*(X).$$
\end{enumerate}
\end{prop}
\begin{proof}
This is an immediate consequence of bivariant formalism, but can also be easily checked on the level of cycles.
\end{proof}

\begin{prop}[Naturality of the duality map]
For every $X \in \Cc$ derived complete intersection, there exists a natural duality morphism
$$- \bullet 1_{X/\Zb} : \Omega^*(X) \to \Omega_\bullet(X).$$
Moreover, 
\begin{enumerate}
\item if $f: X \to Y$ is a projective quasi-smooth morphism, then the square
$$
\begin{tikzcd}
\Omega^*(X) \arrow[]{d}{f_!} \arrow[]{r}{- \bullet 1_{X/\Zb}} & \Omega_\bullet(X) \arrow[]{d}{f_*} \\
\Omega^*(Y) \arrow[]{r}{- \bullet 1_{Y / \Zb}} & \Omega_\bullet(Y)
\end{tikzcd}
$$
commutes;
\item if $f: X \to Y$ is such that $\Lb_{X/Y}$ has Tor-dimension $\leq 1$, then the square
$$
\begin{tikzcd}
\Omega^*(Y) \arrow[]{r}{- \bullet 1_{Y / \Zb}} \arrow[]{d}{f^*} & \Omega_\bullet(Y) \arrow[]{d}{f^!} \\
\Omega^*(X) \arrow[]{r}{- \bullet 1_{X / \Zb}} & \Omega_\bullet(X)
\end{tikzcd}
$$
commutes.
\end{enumerate}
\end{prop}
\begin{proof}
This is an immediate consequence of bivariant formalism, but can also be easily checked on the level of cycles.
\end{proof}

\subsubsection{Basic properties}

Let us then recall the basic properties of algebraic cobordism. Note that the formal group law acting on the Chern classes of $\Omega^*$ induces a map 
$$\Lb^* \to \Omega^*(X)$$
for all $X \in \Cc$, which is compatible with pullbacks. We will have to use some of the basic properties of this morphism later in the article, which we collect below.

\begin{prop}\label{LSmoothProp}
Let $X \in \Cc$. Then the image of $\Lb^* \to \Omega^*(X)$ is generated by derived schemes smooth over $X$.
\end{prop}
\begin{proof}
By the same argument as in \cite{Levine:2007}, one shows that the image of $\Lb^* \to \Omega^*(X)$ is generated by towers of projective bundles over $X$. 
\end{proof}

\begin{prop}\label{LInjectivityProp}
Let $X \in \Cc$. Then the natural map
$$\Lb^* \to \Omega^*(X)$$
is an injection.
\end{prop}
\begin{proof}
Since the morphism is compatible with pullbacks, we can pull back to the generic point of an irreducible component of $X$ to reduce to the case where $X \simeq \Spec(k)$ is the spectrum of a field $k$. Moreover, following the arguments of \cite{Levine:2007} Section 4.1.9, which are based on the ideas of Quillen in \cite{quillen1971} and depend only on the formal properties of Chern classes, we can construct the \emph{total Landweber--Novikov operator}
$$\Omega^* \to \Omega^*[b_1,b_2,...]^{(t)},$$
and combining this with the degree morphism
$$\deg: \Omega^0\big(\Spec(k)\big) \to K^0\big(\Spec(k)\big) \cong \Zb,$$
we obtain a natural morphism
$$\Omega^*\big(\Spec(k)\big) \to \Zb[b_1,b_2,...].$$
As in \cite{Levine:2007} Lemma 4.3.1, one proves that the composition $\Lb^* \to \Zb[b_1,b_2,...]$ is an injection, proving the claim.
\end{proof}

Next we have to recall that the decomposition relations used in the construction of $\Omega^\bullet$ imply the so called snc relations of Lowrey and Schürg, which follow as a special case of the following result.

\begin{prop}[\cite{annala-base-ind-cob} Lemma 3.5]\label{SNCRelationsProp}
Let $X \in \Cc$ be a derived complete intersection scheme, and let 
$$D \simeq n_1 D_1 + \cdots + n_r D_r$$
be a virtual Cartier divisors on $X$ with $n_i > 0$. Let us denote for every $I \subset [r]$ by $\iota^I$ the canonical inclusion of the derived intersection
$$D_I := \bigcap_{i \in I} D_i \hook D.$$
Then
$$1_{D / \Zb} = \sum_{I \subset [r]} \iota^I_*\Bigg( F^{n_1,...,n_r}_I\big(c_1(\Oc(D_1)), ..., c_1(\Oc(D_r))\big) \bullet 1_{D_I / \Zb} \Bigg) \in \Omega_\bullet(D)$$
for universal homogeneous power series
$$F^{n_1,...,n_r}_I(x_1,...,x_r) \in \Lb^*[[x_1,...,x_r]]$$
of degree $1 - \abs{I}$. \qed
\end{prop}

We also record following dimension formula, which can be used to calculate the Krull dimension of $V$ in a nice enough cycle $[V \to X] \in \Omega^{-d}(X)$.

\begin{lem}\label{DegDimLem}
Let $A$ be a discrete valuation ring with residue field $\kappa$, let $V$ and $X$ be integral and projective $A$-schemes and let $V \to X$ be an lci $A$-morphism of relative virtual dimension $d$. Then 
$$\dim(V) = \dim(X) + d.$$
\end{lem}
\begin{proof}
Note that the analogous claim is known over fields, and therefore we obtain the result in the case that $X$ is not flat over $A$ (and hence $X$ and $V$ are projective $\kappa$-schemes). To prove the general case, we will use the dimension formula (see e.g. \cite{stacks} Tag 02JU) to conclude that if $X$ is flat over $A$, then 
\begin{align*}
\dim(X) &= \dim(X_\eta) + 1 \\
&= \dim(X_\kappa) + 1,
\end{align*}
where $X_\eta$ and $X_\kappa$ are the general and the special fibres of $X \to \Spec(A)$ respectively.

If both $V$ and $X$ are flat over $A$, then 
$$
\begin{tikzcd}
V_\kappa \arrow[]{r} \arrow[]{d} & V \arrow[]{d} \\
X_\kappa \arrow[]{r} & X
\end{tikzcd}
$$
is derived Cartesian and therefore $\dim_v(V / X) = \dim_v(V_\kappa / X_\kappa)$, where $\dim_v$ stands for the relative virtual dimension. It follows that
\begin{align*}
\dim(X) &= \dim(X_\kappa) + 1 \\
&= \dim(V_\kappa) + d + 1 \\
&= \dim(V) + d,
\end{align*}
and we are done in this case.

If $X$ is flat over $A$ but $V$ is not, then $V \to X$ factors as 
$$V \to X_\kappa \hook X$$
and $\dim_v(V / X_\kappa) = \dim_v(V / X) + 1$. Therefore
\begin{align*}
\dim(V) &= \dim(X_\kappa) + d + 1 \\
&= \dim(X) + d,
\end{align*}
proving the claim in the last remaining case.
\end{proof}

Finally, we record the following useful formulas, which are going to be useful when classifying low dimensional varieties up to cobordism.

\begin{lem}\label{ClassOfPBLem}
Let $X \in \Cc$, and let $E$ be a vector bundle of rank $r$ on $X$. Then
$$[\Pb^{r-1}_X \to X] - [\Pb(E) - X]$$
is a $\Lb$-linear combination of elements of positive degree in $\Omega^{1-r}(X)$.
\end{lem}
\begin{proof}
By twisting $E$ is necessary, we may assume that there exists a short exact sequence
$$0 \to E \to \Oc^{\oplus N} \to F \to 0$$
of vector bundles. Let $\pi$ be the morphism $\Pb^{N-1}_X \to X$; it follows that
$$[\Pb(E) \to X] = \pi_! \big(c_{N-r}(F(1))\big) \in \Omega^{1-r}(X),$$
is an $\Lb$-linear combination of products of Chern classes of $F$, proving the claim.
\end{proof}

\begin{lem}\label{ClassInBlowUpLem}
Let $Z \hook X$ be a derived regular embedding in $\Cc$. Then
$$1_X - [\bl_Z(X) \to X]$$
is a $\Lb$-linear combination of elements of positive degree in $\Omega^{0}(X)$.
\end{lem}
\begin{proof}
Considering the algebraic cobordism given by the bow up of $\Pb^1 \times X$ at $\infty \times Z$, we obtain the formula
$$1_X - [\bl_Z(X) \to X] = [\Pb_Z(\Nc_{Z/X} \oplus \Oc) \to X] - [\Pb_{\Pb_Z(\Nc_{Z/X})}(\Oc(1) \oplus \Oc) \to X]$$
in $\Omega^0(X)$. The claim follows from applying Lemma \ref{ClassOfPBLem} to the classes $[\Pb_Z(\Nc_{Z/X} \oplus \Oc) \to Z]$ and $[\Pb_{\Pb_Z(\Nc_{Z/X})}(\Oc(1) \oplus \Oc) \to Z]$ and then pushing forward.
\end{proof}

\subsection{Desingularization by alterations}

In this section we are going to recall desingularization results, which will play an important role later in this article. Desingularization by alterations was first introduced by de Jong in \cite{de-jong-alterations}, after which it (and several improvements of the original theorem) have found numerous applications in the study of algebraic geometry in positive and mixed characteristic. Our main reference will be the fairly recent article \cite{temkin-sep} by Temkin.

We start by recalling alterations.

\begin{defn}\label{AlterationDef}
A map $\pi: X \to Y$ of integral Noetherian schemes is called an \emph{alteration} if it is proper, dominant and generically finite. If $\Pc$ is a subset of the prime numbers, then $\pi$ is called a \emph{$\Pc$-alteration} if its degree is a \emph{$\Pc$-number}, i.e., all of its prime divisors lie in $\Pc$. If $\abs{\Pc} \leq 1$, then we will also use the term \emph{$e$-alteration}, where 
$$e = 
\begin{cases}
p \in \Pc & \text{if $\abs{\Pc} = 1$}; \\
1 & \text{if $\Pc = \emptyset$}
\end{cases}
$$ 
is the characteristic exponent of $\Pc$.
\end{defn}

Given an integral scheme $X$, we are going to denote by $\mathrm{char}(X)$ the set of non-zero residual characteristic of $X$. The main result of Temkin is that, under certain assumptions, $X$ admits a $\mathrm{char}(X)$-alteration $X' \to X$ from a regular scheme $X$. Before giving the main result, we recall the following terminology.

\begin{defn}\label{QExDefn}
A Noetherian ring $A$ is \emph{quasi-excellent} if for all primes $\pfr \subset A$, the completion morphism $A_\pfr \to \hat A_\pfr$ is flat (automatic) and has geometrically regular fibres, and for all finite type $A$-algebras $B$, the regular locus of $B$ (i.e., the set of primes $\qfr \subset B$ such that $B_\qfr$ is a regular local ring) is an open subset of $\Spec(B)$. A quasi-excellent ring $A$ is called \emph{excellent} if it is also universally catenary (e.g. $A$ is a regular local ring). A Noetherian scheme $X$ is \emph{(quasi-)excellent} if it admits an open affine cover by spectra of (quasi-)excellent rings.
\end{defn}

The following result is a special case of the main result of \cite{temkin-sep} (combined with \cite{cossart-piltant}).

\begin{thm}
Let $X$ be an integral Noetherian scheme admitting a finite type morphism to a quasi-excellent scheme $Y$ of dimension at most $3$, and let $Z \hook X$ be a closed subscheme. Then there exists a projective $\mathrm{char}(X)$-alteration
$$\pi: X' \to X$$ 
with $X'$ regular and $\pi^{-1}(Z)$ a strict normal crossing divisor. \qed 
\end{thm}

We will only use this result in the case where $Y$ is the spectrum of a field or an excellent discrete valuation ring.

\subsection{Bertini theorems}

Bertini theorems, saying that ample enough line bundles have global sections whose (derived) vanishing locus has good properties, are going to play an important role in the arguments of this paper. We start with the following easy observation, which shows the existence of enough ``classical sections'' in great generality.

\begin{lem}\label{ClassicalBertiniLem}
Let $X$ be a quasi-projective scheme over a Noetherian ring $A$ and let $\Ls$ be an ample line bundle. Then for all $n \gg 0$ the line bundle $\Ls^{\otimes n}$ has a global section $s$ which is not a zerodivisor on $X$. In particular, the derived vanishing locus of $s$ coincides with the classical vanishing locus of $s$.
\end{lem}
\begin{proof}
This is classical, but see for example \cite{gabber-liu-lorenzini} Theorem 5.1 for a reference.
\end{proof}

When $A$ is a discrete valuation ring, we can say a lot more using the recent Bertini-regularity result over discrete valuation rings by Ghosh--Krishna.  

\begin{thm}[Cf. \cite{ghosh-krishna} Theorem 9.6]\label{BertiniRegularityThm}
Let $A$ be a discrete valuation ring, let $X$ be a regular quasi-projective $A$-scheme, and let $\Ls$ be a very ample line bundle on $X$. Then for all $n \gg 0$, the line bundle $\Ls^{\otimes n}$ has a global section $s$ whose vanishing locus $Z_s$ is regular and of codimension 1. Moreover, if $X$ is flat over $A$, then we can find such a global section $s$ with $Z_s$ flat over $A$.
\end{thm}
\begin{proof}
We may assume without loss of generality that $X$ is connected. Note that the authors assume the structure morphism $X \to \Spec(A)$ to be surjective. However, if this is not the case, then $X$ is quasi-projective over either the fraction field of $A$ or the residue field of $A$, and the claim follows from the Bertini-regularity theorems over fields. Moreover, the authors assume the Krull dimension of $X$ to be at least $2$, but this is only because the surjectivity of $X \to \Spec(A)$ implies that otherwise $X$ is going to be affine and semilocal, and hence all line bundles are going to be trivial, rendering the claim trivial. To prove the last claim, we note that if $X$ is flat over $A$, then $Z_s$ is not flat over $A$ if and only if it has a component defined over the residue field of $A$. But looking at the proof of \cite{ghosh-krishna} Theorem 9.6, it is clear that this does not happen, so we are done.
\end{proof}

\section{Presentations of Chern classes and refined projective bundle formulas}\label{RefPBFSect} 

The purpose of this section is to prove that Chern classes of vector bundles and line bundles often admit presentations in terms of nice cobordism cycles, and to use this to prove several refined versions of the projective bundle formula, which will play an important role in Section \ref{AlgebraicSpivakSect}. We note the unfortunate expositional fact that  some results of Section \ref{RefPBFSubSect} for discrete valuation rings use the results of Section \ref{AlgebraicSpivakSect} for fields, so the logic of Sections \ref{RefPBFSect} and \ref{AlgebraicSpivakSect} proceeds really as follows:
\begin{center}
Section \ref{RefPBFSect} for fields $\Rightarrow$ Section \ref{AlgebraicSpivakSect} for fields $\Rightarrow$ Section \ref{RefPBFSect} for discrete valuation rings $\Rightarrow$ Section \ref{AlgebraicSpivakSect} for discrete valuation rings.
\end{center}
We hope that the reader does not get confused because of this nonlinear narrative.

\subsection{Presentations of Chern classes}\label{ChernPresSubSect}

The purpose of this section is to record several presentability results of Chern classes that are going to be useful in the proofs of the refined projective bundle formulas as well as later in the article. We start with the following observation, which is useful when we know that large enough powers of a line bundle admit nice sections.

\begin{lem}\label{PowerChernLem}
Let $X$ be a finite dimensional divisorial Noetherian derived scheme and let $\Ls$ be a line bundle. Then, given coprime integers $p$ and $q$ and integers $a,b \in \Zb$ such that $ap + bq = 1$, we have that 
$$c_1(\Ls) = \big(a c_1(\Ls^{\otimes p}) + b c_1(\Ls^{\otimes q}) \big) \bullet \Bigg(1_X + \sum_{i=1}^\infty b_i c_1(\Ls)^i \Bigg) \in \Omega^1(X),$$
where $b_i := b_i(p,q,a,b) \in \Lb^{-i}$.
\end{lem}
\begin{proof}
Indeed, from the formal group law it follows that
$$a c_1(\Ls^{\otimes p}) + b c_1(\Ls^{\otimes q}) = c_1(\Ls) \bullet \Bigg( 1_X +  \sum_{i=1}^\infty a_i c_1(\Ls)^i \Bigg)$$
where $a_i \in \Lb^{-i}$. The claim follows from the fact that $1_X +  \sum_{i=1}^\infty a_i c_1(\Ls)^i$ is invertible.
\end{proof}

The next result will enable us to find nice presentations of Chern classes of all vector bundles by arguing inductively on the rank.

\begin{lem}\label{FormulaForEtaLem}
Let $X$ be a finite dimensional Noetherian derived scheme having an ample line bundle, and let $E$ be a vector bundle on $X$. Then, denoting by 
$$0 \to \Oc(-1) \to E \to Q \to 0$$
the tautological exact sequence of vector bundles on $\Pb(E)$, the class 
$${ 1_{\Pb(E)} + c_1(\Oc(-1)) \bullet [\Pb_{\Pb(E)}(\Oc(-1) \oplus \Oc) \to \Pb(E)] \over 1 - c_r(E) \bullet [\Pb_{\Pb(E)}(E \oplus \Oc) \to \Pb(E)]} \bullet c_{r-1}(Q) \in \Omega^{r-1}\big(\Pb(E)\big)$$ 
pushes forward to $1_X \in \Omega^0(X)$.
\end{lem}
\begin{proof}
Indeed, this is just \cite{annala-base-ind-cob} Lemma 3.28.
\end{proof}

Our first two results show that Chern classes of vector bundles on classical schemes can be often presented by classical cycles. These results are not needed later in the article, but they might be useful for other purposes, which is why we record them here. The uninterested reader may skip ahead to Lemma \ref{RegLineChernLem}.

We begin with the case of line bundles.

\begin{lem}\label{LCILineChernLem}
Let $X$ be a quasi-projective scheme over a Noetherian ring $A$ and let $\Ls$ be a line bundle. Then $c_1(\Ls) \in \Omega^1(X)$ is equivalent to an $\Lb$-linear combination of cycles of form $[Z \hook X]$, where $Z \hook X$ is a classical regular embedding. 
\end{lem}
\begin{proof}
We start by proving the claim for ample line bundles by arguing inductively on the Krull dimension of $X$, the base case of an empty scheme being obvious. By Lemma \ref{ClassicalBertiniLem} we can find coprime integers $p$ and $q$ and effective Cartier divisors $i_1: Z_1 \hook X$ and $i_2: Z_2 \hook X$ in the linear systems of $\Ls^{\otimes p}$ and $\Ls^{\otimes q}$ respectively. By Lemma \ref{PowerChernLem} and the projection formula it follows that
\begin{align*}
c_1(\Ls) &= \big(a[Z_1 \hook X] + b[Z_2 \hook X] \big) \bullet \Bigg( 1_X +  \sum_{i=1}^\infty b_i c_1(\Ls)^i \Bigg) \\
&= a[Z_1 \hook X] + b[Z_2 \hook X] + a i_{1!} \Bigg(\sum_{i=1}^\infty b_i c_1(\Ls \vert_{Z_1})^i \Bigg) 
+ b i_{2!} \Bigg(\sum_{i=1}^\infty b_i c_1(\Ls \vert_{Z_2})^i \Bigg)
\end{align*}
with $a,b \in \Zb$ and $b_i \in \Lb$, so we have proven the claim for $\Ls$ ample by the inductive assumption.

Suppose then that $\Ls$ is an arbitrary line bundle on $X$. By the quasi-projectivity of $X$ we can find ample line bundles $\Ls_1$ and $\Ls_2$ so that $\Ls \cong \Ls_1 \otimes \Ls_2^\vee$, and therefore it follows from the formal group law that
$$c_1(\Ls) = \sum_{i,j} b_{ij} c_1(\Ls_1)^i \bullet c_1(\Ls_2)^j$$
for some $b_{ij} \in \Lb$. Hence the claim follows from the ample case using the projection formula and the fact that ample line bundles are stable under pullbacks along immersions.
\end{proof}

It is then not very hard to deal with arbitrary vector bundles.

\begin{prop}\label{LCIChernPresentationProp}
Let $X$ be a quasi-projective scheme over a Noetherian ring $A$ and let $E$ be a vector bundle on $X$. Then $c_i(E) \in \Omega^i(X)$ is equivalent to an integral combination of cycles of form $[V \to X]$ with $V \to X$ lci.
\end{prop} 
\begin{proof}
We will proceed by induction on the rank $r$ of the vector bundle, the base case $r = 1$ following from Lemma \ref{LCILineChernLem}. Suppose that $E$ is a rank $r$ vector bundle on $X$ and that the claim is known for all quasi-projective $A$-schemes and for all vector bundles of rank at most $r$. Then, by Lemma \ref{FormulaForEtaLem} and the projection formula, we have that
$$c_i(E) = \pi_! \Bigg( { 1_{\Pb(E)} + c_1(\Oc(-1)) \bullet [\Pb_{\Pb(E)}(\Oc(-1) \oplus \Oc) \to \Pb(E)] \over 1 - c_r(E) \bullet [\Pb_{\Pb(E)}(E \oplus \Oc) \to \Pb(E)]} \bullet c_{r-1}(Q) \bullet c_i(E) \Bigg) \in \Omega^*(X)$$
where $\pi$ is the natural map $\Pb(E) \to X$. Moreover, the element 
$${ 1_{\Pb(E)} + c_1(\Oc(-1)) \bullet [\Pb_{\Pb(E)}(\Oc(-1) \oplus \Oc) \to \Pb(E)] \over 1 - c_r(E) \bullet [\Pb_{\Pb(E)}(E \oplus \Oc) \to \Pb(E)]} \bullet c_{r-1}(Q) \bullet c_i(E)$$
can be expressed in the desired form by the inductive assumption and the Whitney sum formula, so the claim follows by pushing forward.
\end{proof}

Next we show that we can prove stronger results when working over discrete valuation rings. Again, we start with the case of line bundles.

\begin{lem}\label{RegLineChernLem}
Let $X$ be a regular quasi-projective scheme over a discrete valuation ring (or a field) $A$, and let $\Ls$ be a line bundle on $X$. Then $c_1(\Ls) \in \Omega^1(X)$ is equivalent to an $\Lb$-linear combination of cycles of form $[Z \hook X]$, where $Z$ is a regular scheme. If $X$ is flat over $A$, then we can moreover assume the $Z$ to be flat over $A$. 
\end{lem}
\begin{proof}
Let us first assume that $\Ls$ is very ample. We will proceed by induction on the Krull dimension of $X$ the base case of an empty scheme being obvious. By Theorem \ref{BertiniRegularityThm} we can find coprime integers $p$ and $q$ and regular divisors $i_1: Z_1 \hook X$ and $i_2: Z_2 \hook X$ in the linear systems of $\Ls^{\otimes p}$ and $\Ls^{\otimes q}$ respectively. By Lemma \ref{PowerChernLem} and the projection formula it follows that
\begin{align*}
c_1(\Ls) &= \big(a[Z_1 \hook X] + b[Z_2 \hook X] \big) \bullet \Bigg( 1_X +  \sum_{i=1}^\infty b_i c_1(\Ls)^i \Bigg) \\
&= a[Z_1 \hook X] + b[Z_2 \hook X] + a i_{1!} \Bigg(\sum_{i=1}^\infty b_i c_1(\Ls \vert_{Z_1})^i \Bigg) 
+ b i_{2!} \Bigg(\sum_{i=1}^\infty b_i c_1(\Ls \vert_{Z_2})^i \Bigg)
\end{align*}
with $a,b \in \Zb$ and $b_i \in \Lb$, and the claim follows from the inductive assumption. Note that if $X$ was flat over $A$, then we could have chosen $Z_i$ to be flat over $A$ too, proving the stronger claim as well.

In general, we can find very ample line bundles $\Ls_1$ and $\Ls_2$ so that $\Ls \cong \Ls_1 \otimes \Ls_2^\vee$, and therefore, by the formal group law,
$$c_1(\Ls) = \sum_{i,j \geq 0} b_{ij} c_1(\Ls_1)^i \bullet c_1(\Ls_2)^j \in \Omega^*(X)$$
for some $b_{ij} \in \Lb$. As very ample line bundles are stable under restrictions to closed subschemes, the general case follows from the very ample case.
\end{proof}

Finally, we deal with general vector bundles.

\begin{prop}\label{RegularChernPresentationProp}
Let $X$ be a regular quasi-projective scheme over a discrete valuation ring (or a field) $A$, and let $E$ a vector bundle on $X$. Then $c_i(E) \in \Omega^i(X)$ is equivalent to an integral combination of cycles of form $[V \to X]$ with $V$ regular. If $X$ is flat over $A$, then we can moreover assume the $V$ to be flat over $A$.
\end{prop} 
\begin{proof}
We will proceed by induction on the rank $r$ of the vector bundle, the base case $r = 1$ following from Lemma \ref{RegLineChernLem}. Suppose that $E$ is a rank $r$ vector bundle on $X$ and that the claim is known for all quasi-projective $A$-schemes and for all vector bundles of rank at most $r$. Then, by Lemma \ref{FormulaForEtaLem} and the projection formula, we have that
$$c_i(E) = \pi_! \Bigg( { 1_{\Pb(E)} + c_1(\Oc(-1)) \bullet [\Pb_{\Pb(E)}(\Oc(-1) \oplus \Oc) \to \Pb(E)] \over 1 - c_r(E) \bullet [\Pb_{\Pb(E)}(E \oplus \Oc) \to \Pb(E)]} \bullet c_{r-1}(Q) \bullet c_i(E) \Bigg) \in \Omega^*(X)$$
where $\pi$ is the natural map $\Pb(E) \to X$. Moreover, the element 
$${ 1_{\Pb(E)} + c_1(\Oc(-1)) \bullet [\Pb_{\Pb(E)}(\Oc(-1) \oplus \Oc) \to \Pb(E)] \over 1 - c_r(E) \bullet [\Pb_{\Pb(E)}(E \oplus \Oc) \to \Pb(E)]} \bullet c_{r-1}(Q) \bullet c_i(E)$$
can be expressed in the desired form by the inductive assumption and the Whitney sum formula, so the claim follows by pushing forward.
\end{proof}

\subsection{Refined projective bundle formula}\label{RefPBFSubSect}

The purpose of this section is to prove refined versions of projective bundle formula that are going to be useful later in the article. We start with the easier one. Note that the first part of the following result is not needed later in the article.

\begin{prop}\label{FirstRefinedPBFProp}
Let $X$ be a quasi-projective derived scheme over a finite dimensional Noetherian ring $A$, and let $\pi$ be the projection $\Pb^n \times X \to X$. Then a cobordism class 
$$[V \to \Pb^n \times X] \in \Omega^d\big(\Pb^n \times X \big)$$
is equivalent to
$$\sum_{i=0}^{n}  c_1(\Oc(1))^i \bullet \pi^*(\alpha_i)$$
and
\begin{enumerate}
\item if $V$ is a complete intersection scheme, then $\alpha_i \in \Omega^{d-i}(X)$ are equivalent to integral combinations of complete intersection schemes mapping projectively to $X$;
\item if $A$ is a discrete valuation ring (or a field) and $V$ is regular, then $\alpha_i \in \Omega^{d-i}(X)$ are integral combinations of regular schemes mapping projectively to $X$.
\end{enumerate}
\end{prop}
\begin{proof}
By projective bundle formula there exist unique $\alpha_i \in \Omega^*(X)$ so that
$$[V \to \Pb^n \times X] = \sum_{i=0}^{n} c_1(\Oc(1))^i \bullet \pi^*(\alpha_i) \in \Omega_\bullet(\Pb^n \times X)$$
holds; our task is to find the desired presentation for $\alpha_i$. But this is easy: clearly
$$ \pi_* \Big( c_1(\Oc(1))^i \bullet [V \to \Pb^n \times X] \Big) = \alpha_{n-i} + [\Pb^1] \alpha_{n-i-1} + \cdots + [\Pb^{n-i}] \alpha_0,$$
where $[\Pb^i]$ is the class of $\Pb^i$ in $\Lb^{-i}$, so the first claim follows from Proposition \ref{LCIChernPresentationProp} and the second claim follows from Proposition \ref{RegularChernPresentationProp}.
\end{proof}

\begin{rem}
There is a more general version of Proposition \ref{FirstRefinedPBFProp} that holds for all projective bundles and not just the trivial ones. However, the proof is much more complicated, and since we will not need the more general result, we have chose only to prove a special case.
\end{rem}

The following result is one of the crucial results needed in the proof of algebraic Spivak's theorem as it allows us to approximate the class of a generically finite morphisms to a projective space

\begin{thm}\label{SecondRefinedPBFThm}
Let $A$ be a Henselian discrete valuation ring with a perfect residue field (or let $A$ be an arbitrary field), and let $f: V \to \Pb^n_A$ be a projective morphism from an integral regular scheme $V$ of relative virtual dimension 0. If $f$ is generically finite of degree $d \geq 0$, then 
$$[V \to \Pb^n_A] = d + \sum_{i=1}^{n} c_1(\Oc(1))^i \bullet \pi^*(\alpha_i) \in \Omega^0(\Pb^n_A)[e^{-1}]$$
where $\alpha_i$ are integral combinations of regular projective $A$-schemes  and $e$ is the residual characteristic exponent of $A$. If $A$ is a field, then the equality holds without inverting $e$.
\end{thm}
\begin{proof}
Suppose first that $f$ is dominant, which also implies that $V$ is flat over $A$. Using Proposition \ref{FirstRefinedPBFProp} and Lemma \ref{DegDimLem} it follows that
$$[V \to \Pb^n_A] = \sum_{i=0}^{n}  c_1(\Oc(1))^i \bullet \pi^*(\alpha_i)$$
with
$$\alpha_0 = \sum_i n_i [\Spec(B_i) \to \Spec(A)] \in \Omega^0\big(\Spec(A)\big)$$
where $B_i$ are finite, flat, integral and regular $A$-algebras, and $n_i \in \Zb$. Using the specialization morphism of cohomology theories $\Omega^* \to K^0$ one checks that 
$$\sum_i n_i \deg(B_i / A) = d,$$
so the claim follows from Lemma \ref{RegularFiniteDVRClassLem} below.

Suppose then that $f$ is not dominant. If $V$ is flat over $A$ (e.g. if $A$ is a field), then the above proof shows what we want. Otherwise $V$ is a regular $\kappa$-variety, and by Proposition \ref{FirstRefinedPBFProp} and Lemma \ref{DegDimLem}
$$\alpha_0 = \sum_i n_i [C_i \to \Spec(A)] \in \Omega^0\big( \Spec(A)\big)$$ 
with $C_i$ smooth curves over $\kappa$ and $n_i \in \Zb$. Using the computation of $\Omega^{-1}\big(\Spec(\kappa)\big)[e^{-1}]$ from Corollary \ref{CobRingCor} (whose proof only needs this result for fields), we see that 
\begin{align*}
\alpha_0 &= b[\Spec(\kappa) \hook \Spec(A)] \\
&= 0 \in \Omega^0\big( \Spec(\kappa) \big)[e^{-1}],
\end{align*} 
where $b \in \Lb^{-1}[e^{-1}]$. This finishes the proof.
\end{proof}

We needed the following computation of cobordism classes of finite regular algebras in the above proof.

\begin{lem}\label{RegularFiniteDVRClassLem}
Let $A$ be a Henselian discrete valuation ring with a perfect residue field  (or let $A$ be an arbitrary field) and let $B$ be a finite, flat, integral and regular $A$-algebra of degree $d$. Then 
$$[\Spec(B) \to \Spec(A)] = d \in \Omega^0\big(\Spec(A)\big).$$
\end{lem}
\begin{proof}
Note that if $A$ is a field, then $B$ is a finite field extension of $A$ and the claim follows trivially from Lemma \ref{FiniteClassesLem}, as any field extension factors as a composition of primitive field extensions.

Let us then deal with the case where $A$ is a Henselian discrete valuation ring with a perfect residue field $\kappa$. By \cite{stacks} Tag 04GG $B$ is a Henselian discrete valuation ring; let us denote its residue field by $l$. We will denote $n := [l : k]$. Let $\alpha \in l$ be a primitive element over $k$, let $\bar f$ be its minimal polynomial and let $f \in A[x]$ be a monic lift of $f$. It is clear that $A' := A[x]/(f)$ is a Henselian discrete valuation ring of degree $n$ over $A$. Moreover, as $B$ is Henselian and $\alpha$ is a simple root of $\bar f$, we can lift $\alpha$ to a root $\tilde \alpha$ of $f$ in $B$, giving rise to a morphism $\psi: A' \to B$ by sending $x$ to $\tilde \alpha$. We have factored $A \to B$ as 
$$A \stackrel \phi \to A' \stackrel \psi \to B$$
with $\phi$ primitive by construction and $\psi$ primitive by \cite{serre-local-fields} Chapter I Proposition 18, so the claim follows from Lemma \ref{FiniteClassesLem} below.
\end{proof}

The above lemma needed the following result in its proof.

\begin{lem}\label{FiniteClassesLem}
Let $A$ be a Noetherian ring of finite Krull dimension, and let $\psi: A \to B$ be a finite syntomic morphism of degree $d$. If $\psi$ factors as a composition of primitive finite syntomic morphisms, then 
$$[\Spec(B) \to \Spec(A)] = d \in \Omega^0\big(\Spec(A)\big).$$
\end{lem}
\begin{proof}
It is enough to consider the case where $B$ itself is primitive over $A$. Hence
$$B \cong A[x]/(f)$$
for some $f = x^d + a_{d-1} x^{d-1} + \cdots + a_0 \in A[x]$, and therefore $\Spec(B)$ is the derived vanishing locus of
$$F(x,y) = x^d + a_{d-1} x^{d-1} y + \cdots + a_0 y^d \in \Gamma(\Pb^1_A; \Oc(d))$$
and hence
$$[\Spec(B) \to \Spec(A)] = \pi_! \big( c_1(\Oc(d)) \big)$$
where $\pi$ is the projection $\Pb^1_A \to \Spec(A)$. The claim then follows from noticing that $c_1(\Oc(d)) = d c_1(\Oc(1))$ and $\pi_!\big( c_1(\Oc(1))\big) = 1$.
\end{proof}

We are going to use Theorem \ref{SecondRefinedPBFThm} in Section \ref{AlgebraicSpivakSect} in the form of the following corollary.

\begin{cor}\label{ThirdRefinedPBFCor}
Let $A$ be a Henselian discrete valuation ring with a perfect residue field $\kappa$ (or let $A$ be an arbitrary field), and let $f: D \to \Pb^n_A$ be
a projective morphism of relative virtual dimension $0$ from an snc scheme $D$. If $f$ is generically finite of degree $d \geq 0$, then 
$$[D \to \Pb^n_A] = d \bullet 1_{\Pb^n_A / \Zb} + \sum_{i=i}^{n} \pi^*(\alpha_i)  \bullet c_1(\Oc(1))^i \bullet 1_{\Pb^n_A/\Zb} \in \Omega_\bullet(\Pb^n_A)[e^{-1}]$$
with $\alpha_i \in \Omega^{-i}(X)$ integral combinations of regular projective $A$-schemes. If $A$ is a field, then the equality holds without inverting $e$.
\end{cor}
\begin{proof}
Let $X$ be a regular scheme in which $D$ is an snc divisor, let $D_1,...,D_r$ be the prime components of $D$ and let $n_1,...,n_r > 0$ be such that
$$D = n_1 D_1 + \cdots + n_r D_r$$
as effective Cartier divisors on $X$. Then the snc-relations and Poincaré duality imply that
$$[D \to \Pb^n_A] = \sum_{I \subset [r]} f^I_!\Big(F_I^{n_1,...,n_r}\big(c_1(\Oc(D_1)), ..., c_1(\Oc(D_r))\big) \Big) \in \Omega^0(\Pb^n_A)$$
where $f^I$ is the canonical morphism $D_I := \bigcap_{i \in I} D_i \to \Pb^n_A$. Applying Proposition \ref{RegularChernPresentationProp} to the elements
$$F_I^{n_1,...,n_r}\big(c_1(\Oc(D_1)), ..., c_1(\Oc(D_r))\big) \in \Omega^{1 - \abs{I}}(D_I),$$
we conclude that
$$[D \to \Pb^n_A] = n_1 [D_1 \to \Pb^n_A] + \cdots + n_r [D_r \to \Pb^n_A] + \beta \in \Omega^0(\Pb^n_A),$$
where $\beta$ is an integral combination of cycles of form $[V \to \Pb^n_A] \in \Omega^0(\Pb^n_A)$ where $V$ is regular and the morphism $V \to \Pb^n_A$ is non-dominant. Combining this with the fact that
$$d = n_1 \deg(D_1 / \Pb^n_A) + \cdots + n_r \deg(D_r / \Pb^n_A)$$
with $\deg(D_i / \Pb^n_A)=0$ whenever $D_i \to \Pb^n_A$ is non-dominant, the claim follows from Theorem \ref{SecondRefinedPBFThm} and right multiplication by $1_{\Pb^n_A / \Zb}$.
\end{proof}

\section{Algebraic Spivak's theorem and applications}\label{AlgebraicSpivakSect}

The purpose of this section is to prove the algebraic Spivak's theorem (Theorem \ref{GeneralSpivakThm}) and applying it to prove further properties of algebraic bordism. We start by fixing notation in Section \ref{DeformDiagSubSect}, after which we prove Spivak's theorem in Section \ref{ProofOfSpivakSubSect}. Section \ref{ExtendingCyclesSubSect} is dedicated to proving the Extension theorem (Theorem \ref{ExtensionThm}).

Throughout the section $A$ will be either a field or an excellent Henselian discrete valuation ring with a perfect residue field $\kappa$. We will denote by
$$\Omega^A_*(X) := \Omega^{-*}\big( X \to \Spec(A) \big)$$
the \emph{$A$-bordism groups} of quasi-projective derived $A$-schemes $X$, where the grading is given by the relative virtual dimension over $A$. As $A$ is regular, we have the natural isomorphism
$$- \bullet 1_{A / \Zb}: \Omega^A_*(X) \xrightarrow{\cong} \Omega_\bullet(X)$$ 
but note that the right hand side does not have a natural grading. The $A$-bordism groups $\Omega^A_*(X)$ have a canonical $\Omega^*\big(\Spec(A)\big)$-module structure given by the bivariant product, which also has an explicit formula
\begin{align*}
[V \to \Spec(A)] . [W \to X] &= \pi_X^*([V \to \Spec(A)]) \bullet [W \to X] \\
&= [V \times^R_{\Spec(A)} W \to X], 
\end{align*}
where $\pi_X$ is the structure morphism $X \to \Spec(A)$. Pushforwards and Gysin pullbacks are maps of $\Omega^*\big(\Spec(A)\big)$-modules.
 
\subsection{Deformation diagrams}\label{DeformDiagSubSect}

Suppose that $X$ is a connected quasi-smooth and quasi-projective derived $A$-scheme, and let $i: X \hook U$ be a closed embedding with $U$ an open subscheme of $\Pb^n_A$ for some $n \geq 0$. We will denote by 
$$\bar i_\cl: \overline{X}_\cl \hook \Pb^n_A$$
the scheme theoretic closure of the truncation $X_\cl$ of $X$ in $U$. The purpose of this section is to record and study the basic properties of three deformation diagrams playing an important role in Sections \ref{ProofOfSpivakSubSect} and \ref{ExtendingCyclesSubSect}. Let us start with the derived deformation diagram.

\begin{cons}[Derived deformation diagram of $i$]\label{DeformDiag1}
Denoting the derived blow up of $\infty \times X \hook \Pb^1 \times U$ by $M(X/U)$ gives rise to a derived Cartesian diagram
$$
\begin{tikzcd}
X \arrow[hook]{r}{i^\infty} \arrow[hook]{d}{j'^\infty} & \Pb(\Nc_{X/U} \oplus \Oc) + \bl_X(U) \arrow[]{r} \arrow[hook]{d}{j^\infty} & \infty \arrow[hook]{d} \\
\Pb^1 \times X \arrow[hook]{r}{\hat i} & M(X/U) \arrow[]{r} & \Pb^1 \\
X \arrow[hook]{r}{i} \arrow[hook']{u}[swap]{j'^0} & U \arrow[]{r} \arrow[hook']{u}[swap]{j^0} & 0 \arrow[hook']{u}
\end{tikzcd}
$$
where $+$ denotes the sum of virtual Cartier divisors. Let us denote the induced morphisms $\Pb(\Nc_{X/U} \oplus \Oc) \hook M(X/U)$ and $\bl_X(U) \hook M(X/U)$ by $j^\infty_1$ and $j^\infty_2$ respectively.
\end{cons}

Since $\Pb^1 \times X$ does not meet $\bl_X(U)$ by Theorem \ref{BlowUpPropertiesThm}, we have the following result.

\begin{lem}\label{DeformLem}
Let everything be as above. Then 
$$i^! \circ  j^{0!} = s^! \circ j^{\infty!}_1 : \Omega^A_*\big(M(X/U)\big) \to \Omega^A_*(X)$$
where $s$ is the zero section $X \hook \Pb(\Nc_{X/U} \oplus \Oc)$.
\end{lem}
\begin{proof}
This follows immediately from the fact that the squares
$$
\begin{tikzcd}
X \arrow[hook]{r}{s} \arrow[hook]{d}{j'^\infty} & \Pb(\Nc_{X/U} \oplus \Oc) \arrow[hook]{d}{j^\infty_1} \\
\Pb^1 \times X \arrow[hook]{r}{\hat i} & M(X/U)
\end{tikzcd}
\text{ \ and \ }
\begin{tikzcd}
X \arrow[hook]{r}{i} \arrow[hook]{d}{j'^0} & U  \arrow[hook]{d}{j^0} \\
\Pb^1 \times X \arrow[hook]{r}{\hat i} & M(X/U)  
\end{tikzcd}
$$
commute up to homotopy.
\end{proof}

We will combine the above result with the following standard observation in  Section \ref{ProofOfSpivakSubSect}.

\begin{lem}\label{SectionPullbackLem}
Let $X$ be a quasi-projective derived $A$ scheme, let $\rho: \Pb(E \oplus \Oc) \to X$ be a projective bundle over $X$ with $r = \rank(E)$, and let $s: X \hook \Pb(E \oplus \Oc)$ be the zero section. Then 
$$s^!(-) = \rho_* \big(c_r(E(1)) \bullet - \big): \Omega^A_*\big(\Pb(E \oplus \Oc)\big) \to \Omega^A_{*-r}(X).$$
\end{lem}
\begin{proof}
This follows immediately from the fact that $s$ is the derived vanishing locus of a global section of $E(1)$.
\end{proof}

We will also need the following classical deformation diagrams. 

\begin{cons}[Classical deformation diagram of $\bar i_\cl$]\label{DeformDiag2}
Denoting the classical blow up of $\infty \times \overline X_\cl \hook \Pb^1 \times \Pb^n_A$ by $M^\cl(\overline X_\cl /\Pb^n_A)$ gives us the diagram
$$
\begin{tikzcd}
\overline X_\cl \arrow[hook]{r}{\bar i^\infty_\cl} \arrow[hook]{d}{\bar j'^\infty_\cl} & \Ec + \bl^\cl_{\overline X_\cl}(\Pb^n_A) \arrow[]{r} \arrow[hook]{d}{\bar j^\infty_\cl} & \infty \arrow[hook]{d} \\
\Pb^1 \times \overline X_\cl \arrow[hook]{r}{\hat i_\cl} & M^\cl(\overline X_\cl /\Pb^n_A) \arrow[]{r} & \Pb^1 \\
\overline X_\cl \arrow[hook]{r}{\bar i_\cl} \arrow[hook']{u}[swap]{\bar j'^0_\cl} & \Pb^n_A \arrow[]{r}{} \arrow[hook']{u}[swap]{\bar j^0_\cl} & 0 \arrow[hook']{u}
\end{tikzcd}
$$
where $\Ec$ denotes the exceptional divisor of the blow up. We will denote by $\bar j^\infty_{\cl,1}$ the induced morphism $\Ec \hook M^\cl(\overline X_\cl/\Pb^n_A)$.
\end{cons}

\begin{cons}[Classical deformation diagram of $i_\cl$]\label{DeformDiag3}
Restricting the deformation diagram of Construction \ref{DeformDiag2} to the open subscheme $U \subset \Pb^n_A$ gives rise to the diagram
$$
\begin{tikzcd}
 X_\cl \arrow[hook]{r}{i^\infty_\cl} \arrow[hook]{d}{j'^\infty_\cl} & \Ec^\circ + \bl^\cl_{X_\cl}(U) \arrow[]{r} \arrow[hook]{d}{j^\infty_\cl} & \infty \arrow[hook]{d} \\
\Pb^1 \times X_\cl \arrow[hook]{r}{\hat i_\cl^\circ} & M^\cl(X_\cl / U) \arrow[]{r} & \Pb^1 \\
X_\cl \arrow[hook]{r}{i_\cl} \arrow[hook']{u}[swap]{j'^0_\cl} & U \arrow[]{r} \arrow[hook']{u}[swap]{j^0_\cl} & 0 \arrow[hook']{u}
\end{tikzcd}
$$
The induced morphism $\Ec^\circ \hook M^\cl(X_\cl / U)$ is denoted by $j^\infty_{\cl,1}$.
\end{cons}

As a special case of Proposition \ref{BlowUpVsTruncationProp}, we obtain the following result.

\begin{lem}\label{DeformLem2}
The square
$$
\begin{tikzcd}
\Ec^\circ \arrow[hook]{d}{\iota'} \arrow[hook]{r}{j^\infty_{\cl, 1}} & M^\cl(X_\cl / U) \arrow[hook]{d}{\iota} \\
\Pb(\Nc_{X/U} \oplus \Oc) \arrow[hook]{r}{j^\infty_1} & M(X/U)
\end{tikzcd}
$$
is derived Cartesian. \qed
\end{lem}

\subsection{Proof of the algebraic Spivak's theorem}\label{ProofOfSpivakSubSect}

In this section we prove the algebraic Spivak's theorem. For simplicity we will call elements of form
$$[V \to X] \in \Omega^A_*(X)$$
and
$$[V \to \Spec(A)] \in \Omega^*\big(\Spec(A)\big)$$
with $V$ a regular scheme \emph{regular cycles}. The algebraic Spivak's theorem will easily follow from the following result.

\begin{lem}\label{FundClassLem}
Let $X$ be a connected quasi-smooth and quasi-projective derived $A$-scheme of virtual $A$-dimension $d$. Then the fundamental class
$$1_{X/A} \in \Omega^A_d(X)[e^{-1}]$$
is equivalent to a $\Zb[e^{-1}]$-linear combination of elements of form 
$$(\alpha_1 \bullet \cdots \bullet \alpha_s) . \beta$$
with $\alpha_i \in \Omega^*\big( \Spec(A) \big)$ and $\beta \in \Omega^A_*(X)$ integral combinations of regular cycles. 
\end{lem}

Throughout this section, we will use the notation of Section \ref{DeformDiagSubSect}. Let us start with the following construction.

\begin{cons}\label{DesingAltOfDeformCons}
Choose a desingularization by an $e$-alteration
$$\pi: W \to M^\cl(\overline X_\cl / \Pb^n_A)$$
with $\Ec' := \pi^{-1}(\Ec)$ and $W_0 := \pi^{-1}(\Pb^n_A)$ snc divisors (here $\Pb^n_A$ is the copy of the projective space over $0$), and denote by
$$\pi^\circ: W^\circ \to M^\cl(X_\cl / U)$$
$\Ec'^\circ$ and $W_0^\circ$ the pullbacks of $\pi$, $\Ec'$ and $W'_0$ along $U \subset \Pb^n_A$. Note that since 
$$
\begin{tikzcd}
W^\circ_0 \arrow[hook]{r} \arrow[]{d} & W^\circ \arrow[]{d} \\
U \arrow[hook]{r} & M(X/U)
\end{tikzcd}
\ \ \  \text{and} \ \ \ 
\begin{tikzcd}
\Ec'^\circ \arrow[hook]{r} \arrow[]{d} & W^\circ \arrow[]{d} \\
\Pb(\Nc_{X/U} \oplus \Oc) \arrow[hook]{r} & M(X/U)
\end{tikzcd} 
$$
are derived Cartesian, it follows from Lemma \ref{DeformLem} that
\begin{equation}\label{MainEq}
i^!([W^\circ_0 \to U]) = s^!\big([\Ec'^\circ \to \Pb(\Nc_{X/U} \oplus \Oc)] \big)
\end{equation}
in $\Omega^A_*(X)$.
\end{cons}

We then have the following lemmas.

\begin{lem}\label{MainEqLem1}
We have that
$$[W^\circ_0 \to U] = e^m \cdot 1_{U/A} + \sum_{i=1}^n \alpha_i . \big(c_1(\Oc_U(1))^i \bullet 1_{U / A} \big) \in \Omega^A_*(U),$$
where $m$ is a non-negative integer, $\Oc_U(1)$ is the restriction of $\Oc(1)$ to $U$, and $\alpha_i \in \Omega^{-i}\big( \Spec(A) \big)$ are integral combinations of regular cycles. 
\end{lem}
\begin{proof}
By construction $\pi: W \to M^\cl(\overline X_\cl/\Pb^n_A)$ is surjective and generically finite of degree $e^m$ for some $m \in \Zb$ non-negative (and therefore also of relative virtual dimension 0 over the regular locus of $M^\cl(\overline X_\cl/\Pb^n_A)$). As $M^\cl(\overline X_\cl/\Pb^n_A)$ is regular away from the fibre over $\infty$, as the inclusion $\Pb^n_A \hook M^\cl(\overline X_\cl/\Pb^n_A)$ of the fibre over 0 is a prime divisor, and as a surjective morphism from an integral scheme to the spectrum of a discrete valuation ring is necessarily flat, it follows that also the morphism $W_0 \to \Pb^n_A$ is generically finite and of degree $e^m$. As $W_0$ is an snc scheme, the claim follows from Corollary \ref{ThirdRefinedPBFCor} and restricting along the inclusion $U \subset \Pb^n_A$.
\end{proof}

Moreover, the other term of equation (\ref{MainEq}) is of the correct form by the following result. 

\begin{lem}\label{MainEqLem2}
We have that 
$$s^!\big([\Ec'^\circ \to \Pb(\Nc_{X/U} \oplus \Oc)] \big) \in \Omega^A_*(X)$$
is an integral combination of regular cycles. 
\end{lem}
\begin{proof}
By Lemma \ref{SectionPullbackLem} we have that
$$s^!\big([\Ec'^\circ \to \Pb(\Nc_{X/U} \oplus \Oc)] \big) = \rho_*\big(c_r(\Nc_{X/U}(1)) \bullet [\Ec'^\circ \to \Pb(\Nc_{X/U} \oplus \Oc)] \big) \in \Omega^A_*(X),$$
where $\rho$ is the canonical projection $\Pb(\Nc_{X/U} \oplus \Oc) \to X$. Moreover, denoting by $\Ec'^\circ_1,...,\Ec'^\circ_r$ the prime components of $\Ec'^\circ$, letting $n_1,...,n_r > 0$ be such that
$$\Ec'^\circ = n_1 \Ec'^\circ_1 + \cdots + n_r \Ec'^\circ_r$$
as effective Cartier divisors on $W^\circ$, and denoting for each $I \subset [r]$ the natural inclusion $\Ec'^\circ_I := \bigcap_{i \in I} \Ec'^\circ_i \hook \Ec'^\circ$, the snc relations imply that
\begin{align*}
c_r(\Nc_{X/U}(1)) \bullet 1_{\Ec'^\circ / A} &= c_r(\Nc_{X/U}(1)) \bullet \sum_{I \subset [r]} \iota^I_*\Big(F_I^{n_1,...,n_r}\big(c_1(\Oc(D_1)), ..., c_1(\Oc(D_r))\big) \bullet 1_{\Ec'^\circ_I / A} \Big) \\
&= \sum_{I \subset [r]} \iota^I_*\Big(c_r(\Nc_{X/U}(1)) \bullet F_I^{n_1,...,n_r}\big(c_1(\Oc(D_1)), ..., c_1(\Oc(D_r))\big) \bullet 1_{\Ec'^\circ_I / A} \Big)
\end{align*}
in $\Omega^A_*(\Ec'^\circ)$, so the claim follows from Proposition \ref{RegularChernPresentationProp} and pushing forward to $X$.
\end{proof}

We are now ready to prove the main lemma needed for the proof of algebraic Spivak's theorem.

\begin{proof}[Proof of Lemma \ref{FundClassLem}]
We will proceed by induction on the Krull dimension of $X$, the base case of an empty scheme being trivial. Suppose that $X$ has Krull dimension $d$ and the result has been proven for derived quasi-projective and quasi-smooth derived $A$ schemes of Krull dimension at most $d-1$. Denoting by $\pi_X$ the structure morphism $X \to \Spec(A)$ and by $\Oc_X(1)$ the restriction of $\Oc(1)$ to $X$, and combining equation (\ref{MainEq}) with Lemmas \ref{MainEqLem1} and \ref{MainEqLem2}, we see that
$$1_{X/A} = {\beta - \sum_{i=1}^n \alpha_i . \big( c_1(\Oc_X(1))^i \bullet 1_{X / A} \big) \over e^m} \in \Omega^A_*(X)[e^{-1}]$$ 
where $\beta \in \Omega^A_*(X)[e^{-1}]$ and $\alpha_i \in \Omega^{-i}\big(\Spec(A)\big)$ are integral combinations  of regular cycles. By Lemma \ref{ChernDimLem} the classes 
$$c_1(\Oc_X(1))^i \bullet 1_{X / A}$$
for $i > 0$ are equivalent to an $\Lb$-linear combination of pushforwards of fundamental classes of derived schemes of Krull dimension at most $d-1$, and therefore these classes can be dealt with inductively. We can therefore express $1_{X / A}$ in the desired form, proving the claim. 
\end{proof}

We used the following lemma in the above proof.

\begin{lem}\label{ChernDimLem}
Let $X$ be a quasi-projective derived $A$-scheme of Krull dimension $d$, and let $\Ls$ be a line bundle on $X$. Then 
$$c_1(\Ls)^i = \sum_{j=1}^N a_j [V_j \hook X] \in \Omega^i(X),$$
where $a_j \in \Lb$ and $V_j$ are derived schemes of Krull dimension at most $d-1$.
\end{lem}
\begin{proof}
It is clearly enough to prove the claim for $i=1$, because the extra copies of $c_1(\Ls)$ can be evaluated as derived vanishing loci of zero sections. Suppose first that $\Ls$ is a very ample line bundle. Then for all $n \gg 0$ the line bundle $\Ls^{\otimes n}$ has a global section $s$ whose vanishing locus does not contain any of the irreducible components of $X_\cl$. In particular we can find coprime integers $p$ and $q$ and global sections $s_p \in \Gamma(X;\Ls^{\otimes p})$ and $s_q \in \Gamma(X; \Ls^{\otimes q})$ whose derived vanishing loci $[Z_1 \hook X]$ and $[Z_2 \hook X]$ have Krull dimension at most $d-1$, and therefore Lemma \ref{PowerChernLem} implies that
$$c_1(\Ls) = \Bigg(1_X + \sum_{i=1}^\infty b_i c_1(\Ls)^i \Bigg) \bullet \big(a [Z_1 \hook X] + b [Z_2 \hook X] \big) \in \Omega^1(X)$$
for $a$ and $b$ integers and $b_i \in \Lb$. Hence the claim follows for $\Ls$ very ample. As a general line bundle $\Ls$ is equivalent to $\Ls_1 \otimes \Ls_2^\vee$ with $\Ls_1$ and $\Ls_2$ very ample, the general case follows using the formal group law.
\end{proof}

It is now easy to prove the algebraic Spivak's theorem.

\begin{thm}[Algebraic Spivak's theorem]\label{GeneralSpivakThm}
Let $A$ be a field or an excellent Henselian discrete valuation ring with a perfect residue field $\kappa$, and let $e$ be the residual characteristic exponent of $A$. Then the algebraic cobordism ring $\Omega^*\big(\Spec(A)\big)[e^{-1}]$ is generated as a $\Zb[e^{-1}]$-algebra by regular projective $A$-schemes. Moreover, for all quasi-projective derived $A$-schemes $X$, $\Omega_\bullet(X)[e^{-1}]$ is generated as an $\Omega^*\big(\Spec(A)\big)[e^{-1}]$-module by classes of regular schemes mapping projectively to $X$.
\end{thm}
\begin{proof}
Suppose $X$ is a quasi-projective derived $A$-scheme. Then $\Omega_\bullet(X)[e^{-1}]$ is by definition generated as a $\Zb[e^{-1}]$-module by cycles of form 
$$[V \stackrel f \to  X] = f_*(1_{V/\Zb})$$
with $f$ projective and $V$ a derived complete intersection scheme, and the claim follows from Lemma \ref{FundClassLem} under the natural isomorphism
$- \bullet 1_{A / \Zb}: \Omega^A_*(X) \stackrel{\cong}{\to} \Omega_\bullet(X)$.
\end{proof}

This result has several immediate corollaries. Let us start with the following.

\begin{cor}\label{StrongerSpivakCor}
Suppose $A$ is as in Theorem \ref{GeneralSpivakThm}. Then, for all quasi-projective derived $A$-schemes $X$, $\Omega_\bullet(X)[e^{-1}]$ is generated as a $\Zb[e^{-1}]$-module by cycles of form
$$[V \to X]$$
with $V$ a classical complete intersection scheme and the structure morphism $V \to \Spec(A)$ is either flat or factors through the unique closed point of $\Spec(A)$.

\end{cor}
\begin{proof}
Let $X$ be a quasi-projective derived $A$-scheme. Then, by Theorem \ref{GeneralSpivakThm}, the algebraic bordism group $\Omega_\bullet(X)$ is generated by $\Zb[e^{-1}]$ linear combinations of elements of form
\begin{align*}
&([V_1 \to \Spec(A)] \bullet \cdots \bullet [V_1 \to \Spec(A)]) . [W \to X] \\
=&   [V_1 \times^R_{\Spec(A)} \cdots \times^R_{\Spec(A)} V_r \times^R_{\Spec(A)} W \to X],
\end{align*}
where $V_i$ and $W$ are regular. If $A$ is a field, then the derived fibre product is classical, proving the claim for fields. 

We are left with proving the first claim for discrete valuation rings. We first observe that since $V_i$ and $W$ are integral, each of them is either flat over $A$ or factors through the residue field $\kappa$ of $A$. If all or all but one of the $V_1,...,V_r,W$ are flat over $A$, the derived fibre product is classical, so we are left with the case where at least two of them factor through $\kappa$. Without loss of generality, we can assume them to be $V = V_1$ and $W$. But then one computes that the derived fibre product $V \times^R_{\Spec(A)} W$ is the derived vanishing locus of the zero section of the trivial line bundle in $V \times^R_{\Spec(\kappa)} W \simeq V \times_{\Spec(\kappa)} W$
and the claim follows from the fact that
\begin{align*}
[V \times^R_{\Spec(A)} W \to X] &= f_* \big(c_1(\Oc) \bullet 1_{V \times_{\Spec(\kappa)} W / \Zb} \big) \\
&= 0 \in \Omega_\bullet(X)
\end{align*}
where $f$ is the induced morphism $V \times_{\Spec(\kappa)} W \to X$.
\end{proof}

\begin{cor}\label{PerfectSpivakCor}
If $k$ is a perfect field, then the $e$-inverted bordism groups $\Omega_\bullet(X)[e^{-1}]$, where $X$ is a quasi-projective derived $k$-scheme, are generated as $\Zb[e^{-1}]$-modules by classes of smooth $k$-varieties mapping projectively to $X$.
\end{cor}
\begin{proof}
This is an immediate corollary of Theorem \ref{GeneralSpivakThm} because a quasi-projective derived $k$-scheme is regular if and only if it is smooth over $k$, and smoothness is stable under derived base change.
\end{proof}

An immediate corollary of the above is the following vanishing result.

\begin{cor}\label{VanishCor}
Let $X$ be a quasi-projective $A$-variety. Then
\begin{enumerate}
\item if $A = k$ is a field, $\Omega^k_i(X)[e^{-1}] = 0$ for all $i<0$;
\item if $A$ is a discrete valuation ring, $\Omega^A_i(X)[e^{-1}] = 0$ for all $i < -1$.
\end{enumerate}
\end{cor}
\begin{proof}
This follows from Corollary \ref{StrongerSpivakCor} because lci $A$-schemes can not have arbitrarily negative relative virtual dimension.
\end{proof}

Note that the above bounds are strict as the fundamental class $1_{\kappa/A}$ of  $\Spec(\kappa)$, where $\kappa$ is the residue field of a discrete valuation ring $A$, lies in degree $-1$. The above vanishing result also has the following amusing corollary, which seems to be difficult to prove directly (although there should be an easier proof based on virtual fundamental classes in intersection theory and Grothendieck--Riemann--Roch formulas).

\begin{cor}\label{CoherentECCor}
Let $k$ be a field, and let $X$ be a projective derived $k$-scheme with negative virtual dimension. Then
$$\sum_{i \in \Zb} (-1)^i \dim_k\big(H^i(X; \Oc_X)\big) = 0 \in \Zb.$$
\end{cor}
\begin{proof}
Indeed, the left hand side is the image of 
$$[X \to \Spec(k)] \in \Omega_*\big(\Spec(k)\big)[e^{-1}]$$
under the natural morphism 
$$\Omega_*\big(\Spec(k)\big)[e^{-1}] \to K^0\big(\Spec(k)\big)[e^{-1}] \cong \Zb[e^{-1}],$$
so the claim follows from Corollary \ref{VanishCor}.
\end{proof}

Let us then turn to the structure of the algebraic cobordism ring of $A$. Using birational geometry of varieties of dimension $\leq 2$, we are able to make modest progress towards Conjecture \ref{CobordismOfLocalRingsConj}.

\begin{cor}\label{CobRingCor}
Let $A$ be as in Theorem \ref{GeneralSpivakThm}. Then 
\begin{enumerate}
\item $\Omega^i\big( \Spec(A) \big)[e^{-1}]$ vanishes for $i>0$, and the natural map 
$$\Zb[e^{-1}] \cong \Lb^0[e^{-1}] \to \Omega^0\big( \Spec(k) \big)[e^{-1}]$$
is an isomorphism; in other words $\Omega^0\big( \Spec(A) \big)$ is the free $\Zb[e^{-1}]$-module generated by $1_{A}$;

\item if $A = k$ is a field, then the natural map 
$$\Lb^*[e^{-1}] \to \Omega^*\big( \Spec(k) \big)[e^{-1}]$$
is an isomorphism in degrees $0, -1$ and $-2$; in other words $\Omega^1\big( \Spec(k) \big)$ is the free $\Zb[e^{-1}]$-module generated by $[\Pb^1_k]$ and $\Omega^2 \big( \Spec(k) \big)$ is the free $\Zb[e^{-1}]$-module generated by $[\Pb^1_k]^2$ and $[\Sigma_{1,k}]$, where $\Sigma_{1,k}$ is the Hirzebruch surface of degree 1 over $k$.
\end{enumerate}
\end{cor}
\begin{proof}
\begin{enumerate}
\item Let us first assume that $A = k$ is a field. The vanishing of $\Omega^i\big( \Spec(A) \big)[e^{-1}]$ for $i > 0$ follows from Corollary \ref{VanishCor}, and by Theorem \ref{GeneralSpivakThm} $\Omega^0\big( \Spec(A) \big)[e^{-1}]$ is generated as a $\Zb[e^{-1}]$-algebra by classes of field extensions of $k$. It then follows from Lemma \ref{RegularFiniteDVRClassLem} that the natural morphism
$$\Zb[e^{-1}] \cong \Lb^0[e^{-1}] \to \Omega^0\big( \Spec(k)\big)[e^{-1}]$$
is surjective, and hence an isomorphism by Prop \ref{LInjectivityProp}.

Suppose then that $A$ is an excellent Henselian discrete valuation ring with a perfect residue field $\kappa$. Then the vanishing of $\Omega^i\big( \Spec(A) \big)[e^{-1}]$ for $i > 1$ follows from Corollary \ref{VanishCor}. Moreover, by Corollary \ref{StrongerSpivakCor}, the Gysin pushforward map
$$\Zb[e^{-1}] \cong \Omega^0 \big( \Spec(\kappa) \big)[e^{-1}] \to \Omega^1 \big( \Spec(A)\big)[e^{-1}]$$
is a surjection because syntomic $A$-schemes cannot have negative relative virtual dimension. As $[\Spec(\kappa) \hook \Spec(A)] = 0$, it follows that $\Omega^1\big(\Spec(A)\big)[e^{-1}]$ vanishes. 

We are left to compute $\Omega^0 \big( \Spec(A)\big)[e^{-1}]$. From the previous paragraph (and Lemma \ref{DegDimLem}) it follows that it is generated as a $\Zb[e^{-1}]$-algebra by classes of form $[V \to \Spec(A)]$, where $V$ is a $1$-dimensional regular and integral projective $A$-scheme. If $V$ is flat over $A$, then $[V \to \Spec(A)]$ lies in the image of the natural morphism $\Lb^0[e^{-1}] \to \Omega^0 \big( \Spec(A)\big)[e^{-1}]$ by Lemma \ref{RegularFiniteDVRClassLem}. On the other hand, if $V$ is not flat over $A$, then it is a smooth curve over $\kappa$. Using the second part of this result, we have that
$$[V \to \Spec(\kappa)] = b \in \Omega^{-1}\big(\Spec(\kappa)\big)[e^{-1}]$$
where $b$ comes from $\Lb^{-1}[e^{-1}]$, and therefore 
\begin{align*}
[V \to \Spec(A)] &= b [\Spec(\kappa) \hook \Spec(A)] \\
&= 0  \in \Omega^{0}\big(\Spec(A)\big)[e^{-1}].
\end{align*}
It follows that the natural map
$$\Zb[e^{-1}] \cong \Lb^0[e^{-1}] \to \Omega^0\big( \Spec(A)\big)[e^{-1}]$$ 
is surjective, and hence an isomorphism by Proposition \ref{LInjectivityProp}.

\item The proof splits into three cases.
\begin{enumerate}
\item \emph{$k$ is infinite and perfect}: We follow the argument of Levine and Morel from \cite{Levine:2007} Section 4.3. Let $X$ be a smooth $k$-variety of dimension $n \leq 2$, and suppose that the claim is known up to dimension $n-1$. It is well known that $X$ admits a finite birational morphism $X \to X'$, where $X' \hook \Pb^{n+1}_k$ is a hypersurface; let us denote the degree of $X'$ by $d$. By the embedded resolution of surfaces (see \cite{Abhyankar1966} Section 0), we may find a morphism
$$\pi: Y \to \Pb^{n+1}_k$$
which is a composition of finitely many blow ups whose centers are smooth subvarieties of dimension $\leq n-1$ of $\Pb^{n+1}_k$ so that
$$\pi^{-1} (X') = \wtil{X} + n_1 \Ec_1 + \cdots + n_r \Ec_r$$
as divisors, where the strict transform $\wtil{X}$ of $X'$ is smooth, $\Ec_i$ are the exceptional divisors (which are projective bundles over smooth varieties of dimension $\leq n-1$), and where $n_i > 0$. As $X$ is the normalization of $X'$, we obtain a canonical morphism
$$\rho: \wtil X \to X.$$
Notice that if $n = 1$, then $\rho$ is an isomorphism, and if $n = 2$, then $\rho$ is a composition of blow ups at closed points (see e.g. \cite{stacks} Tag 0C5R). It follows from Lemma \ref{ClassInBlowUpLem} that $[X]$ lies in the image of $\Lb^*[e^{-1}]$ if and only if $[\wtil X]$ does.

Note that it follows from Lemma \ref{ClassInBlowUpLem} and Corollary \ref{PerfectSpivakCor} that 
$$c_1(\Oc(d)) \bullet [Y \to \Pb^{n+1}_k] = [X' \to \Pb^{n+1}_k] + \sum_{i=1}^m a_i [V_i \to \Pb^{d+1}_k]$$
where $a_i \in \Lb^*$ and $V_i$ are smooth varieties of dimension $\leq n-1$, and as $[X']$ lies in the image of $\Lb^*[e^{-1}]$ (this follows from the formal group law and the fact that $[\Pb^i_k]$ are in the image of $\Lb^*$), it follows that $[\pi^{-1} (X')]$ lies in the image of $\Lb^*[e^{-1}]$. On the other hand, by Lemma \ref{ClassOfPBLem} $[\Ec_i]$ lie in the image of $\Lb^*[e^{-1}]$, and therefore, computing the class of $[\pi^{-1} (X') \hook Y]$ using the formal group law, we see that also $[\wtil X]$ lies in the image of $\Lb^*[e^{-1}]$. But as we already noted, this implies that $[X]$ lies in the image of $\Lb^*[e^{-1}]$, so we are done.

\item \emph{$k$ is not perfect}: Let $X$ be a regular $k$-variety of dimension $n \leq 2$, and let $k^{\mathrm{perf}}$ be the perfection of $k$. It follows from (a) that $[X_{k^{\mathrm{perf}}}]$ lies in the image of $\Lb^*[e^{-1}]$. Moreover, by chasing coefficients we see that this is true for some finite intermediate field extension
$$k \subset k' \subset k^{\mathrm{perf}},$$
and therefore
$$[X] = e^{-m}[X_{k'}]$$
lies in the image of $\Lb^*[e^{-1}]$, where $e^m = [k' : k]$.

\item \emph{$k$ is finite}:  The problem is the strategy of (a) in this case is that a smooth projective $k$-variety $V$ might not admit a finite birational projection to a hypersurface. But this is easy to fix: clearly every such a variety admits such a projection for all large enough finite extensions of $k$, so in particular we can find extensions $k'$ and $k''$ of coprime degrees $p$ and $q$, such that $X_{k'}$ and $X_{k''}$ admit the desired projections over $k'$ and $k''$ respectively. Therefore, we can run the argument of part (a) for $X_{k'}$ and $X_{k''}$, and therefore
\begin{align*}
a [X_{k'} \to \Spec(k)] + b [X_{k''} \to \Spec(k)] &= a [X_{k'} \to \Spec(k)] + b [X_{k''} \to \Spec(k)] \\
&= (ap + bq) [X \to \Spec(k)] \\
&= [X \to \Spec(k)]
\end{align*}
lies in the image of $\Lb^*[e^{-1}] \to \Omega^*\big(\Spec(k)\big)[e^{-1}]$, where $a, b \in \Zb$ are such that $ap + bq = 1$. \qedhere
\end{enumerate}
\end{enumerate}
\end{proof}

Let us also record the following result, concerning the generators of the cobordism ring of a discrete valuation ring.

\begin{prop}\label{GeneratorProp}
Let $A$ be an excellent Henselian discrete valuation ring with a perfect residue field $\kappa$. Then
\begin{enumerate}
\item $\Omega^{-1}\big( \Spec(A) \big)[e^{-1}]$ is generated as a $\Zb[e^{-1}]$-module by regular and flat projective $A$-schemes;

\item if $\kappa$ has characteristic 0, then $\Omega^*\big( \Spec(A) \big)$ is generated as an Abelian group by the classes of syntomic projective $A$-schemes.
\end{enumerate}
\end{prop}
\begin{proof}
In both cases the proof is essentially the same. By Corollary \ref{StrongerSpivakCor}, we see that $\Omega^{i}\big( \Spec(A) \big)[e^{-1}]$ is generated by cycles of form $[V \to \Spec(A)]$ with $V$ a complete intersection scheme that is either flat over $A$ or factors through $\Spec(\kappa)$, and we have to show that non-flat cycles vanish at least in certain degrees. Note that if $V$ is not flat, then
$$[V \to \Spec(\kappa)] \in \Omega^{-d}\big(\Spec(\kappa)\big)[e^{-1}]$$
lies in the image of the natural morphism $\Lb^{-d}[e^{-1}] \to \Omega^{-d}\big(\Spec(\kappa)\big)[e^{-1}]$ if $d$ is at most $2$ or if $\kappa$ has characteristic 0 (\cite{Levine:2007} Theorem 4.3.7), and therefore
\begin{align*}
[V \to \Spec(A)] &= b[\Spec(\kappa) \hook \Spec(A)] \\
&= 0 \in \Omega^{1-d}(A)
\end{align*}
for some $b \in \Lb^{-d}[e^{-1}]$. Hence $\Omega^{-i}\big( \Spec(A) \big)$ is generated by classes of syntomic projective $A$-schemes if $i$ is at most $1$ or if $\kappa$ has characteristic 0. Moreover, as $\Zb[e^{-1}] \cong \Omega^0\big(\Spec(A)\big)[e^{-1}]$, it follows from Theorem \ref{GeneralSpivakThm} that $\Omega^{-1}\big( \Spec(A) \big)[e^{-1}]$ is generated as a $\Zb[e^{-1}]$-module by regular cycles, proving the first claim.
\end{proof}

\subsection{Proof of the Extension theorem}\label{ExtendingCyclesSubSect}

The goal of this section is to prove the following result.

\begin{thm}[Extension theorem]\label{ExtensionThm}
Let $A$ be a field or an excellent Henselian discrete valuation ring with a perfect residue field $\kappa$, and let $j: X \hook \overline X$ be an open embedding of quasi-projective derived $A$-schemes. Then the pullback morphism
$$j^!: \Omega^A_*(\overline X)[e^{-1}] \to \Omega^A_*(X)[e^{-1}]$$
is surjective.
\end{thm}
\noindent As $j^!$ is $\Omega^*\big(\Spec(A)\big)[e^{-1}]$-linear, it is enough by Theorem \ref{GeneralSpivakThm} to show that an element
$$[V \to X] \in \Omega^A_d(X)$$
with $V$ a regular scheme lies in the image of $j^!$. Letting $\overline V$ be the scheme theoretic closure of $V$ in some locally closed embedding
$$V \hook \Pb^n \times X_\cl \subset \Pb^n \times \overline X_\cl$$
which clearly maps projectively to $\overline X$, it is enough to extend the fundamental class of $V$ to an element of $\Omega^A_*(\overline V)[e^{-1}]$. As all quasi-projective $A$-schemes admit a projective compactification, Theorem \ref{ExtensionThm} follows from an easy inductive argument combined with the following lemma and Corollary \ref{VanishCor}. 

\begin{lem}\label{MainExtensionLem}
Let $\overline X$ be a projective $A$-scheme and let $j: X \hook \overline X$ be an open embedding with $X$ a regular scheme of virtual $A$-dimension $d$. If the Gysin pullback morphism
$$j^!: \Omega^A_i(\overline X)[e^{-1}] \to \Omega^A_i(X)[e^{-1}]$$
is surjective for $i \leq d-1$, then there exists a class $\alpha \in \Omega^A_d( \overline X )[e^{-1}]$ with $j^!(\alpha) = 1_{X/A}$.
\end{lem}
\begin{proof}
Note that we can find a (derived) Cartesian square
$$
\begin{tikzcd}
X \arrow[hook]{d}{i} \arrow[hook]{r}{j} & \overline X \arrow[hook]{d}{\bar i} \\
U \arrow[hook]{r} & \Pb^n_A
\end{tikzcd}
$$
where the vertical arrows are closed embeddings and the horizontal arrows are open embeddings. Since $i$ is a (derived) regular embedding, we can apply Construction \ref{DesingAltOfDeformCons}, and conclude that, in the notation of Section \ref{DeformDiagSubSect} and Construction \ref{DesingAltOfDeformCons}, the equality
$$i^!([W^\circ_0 \to U]) = s^!\big([\Ec'^\circ \to \Pb(\Nc_{X/U} \oplus \Oc)] \big) \in \Omega^A_*(X)$$
holds. By Lemma \ref{MainEqLem1} we have that
$$1_{X/A}  = { s^!\big([\Ec'^\circ \to \Pb(\Nc_{X/U} \oplus \Oc)] \big) - \sum_{i=1}^n \alpha_i . \big( c_1(\Oc_X(1))^i \bullet 1_{X / A} \big) \over e^m} \in \Omega^A_*(X)[e^{-1}],$$
where $\Oc_X(1)$ is the restriction of $\Oc(1)$ on $X$, and $\alpha_i \in \Omega^{-i}\big(\Spec(A)\big)$. As $c_1(\Oc_X(1))^i \bullet 1_{X / A} \in \Omega^A_{d-i}(X)[e^{-1}]$, we have reduced the problem to showing that 
$$s^!\big([\Ec'^\circ \to \Pb(\Nc_{X/U} \oplus \Oc)] \big)$$ 
extends. By Lemma \ref{SectionPullbackLem}
$$s^!\big([\Ec'^\circ \to \Pb(\Nc_{X/U} \oplus \Oc)] \big) = \rho_* \big( c_r(\Nc_{X/U}(1)) \bullet [\Ec'^\circ \to \Pb(\Nc_{X/U} \oplus \Oc)] \big) \in \Omega^A_*(X),$$
where $\rho$ is the projection $\Pb(\Nc_{X/U} \oplus \Oc) \to X$, and we can find the desired extension by applying Lemma \ref{ExtendingOnSNCLem} to the open immersion $\Ec'^\circ \hook \Ec'$ and pushing forward to $\overline X$.
\end{proof}

We needed the following lemma in the above proof.

\begin{lem}\label{ExtendingOnSNCLem}
Let $j: D^\circ \hook D$ be an open immersion of snc-schemes admitting an ample line bundle, and let $E$ be a vector bundle on $D^\circ$. Then for all $i \geq 0$ there exists a class $\alpha \in \Omega_\bullet(D)$ so that
$$j^!(\alpha) = c_i(E) \bullet 1_{D^\circ / \Zb} \in \Omega_\bullet(D^\circ).$$ 
\end{lem}
\begin{proof}
Let $D_1,...,D_r$ be the prime components of $D$, and denote for each $I \subset [r]$ by $\iota^I$ the closed embedding 
$$D_I := \bigcap_{i \in I} D_i \hook D.$$
By the snc-relations we have that
$$1_{D/\Zb} = \sum_{I \subset [r]} \iota^I_*(\alpha_I) \in \Omega_\bullet(D),$$
and moreover, if $\Ifr$ contains all $I$ so that $D_I$ meets the image of $j$, then
$$1_{D^\circ / \Zb} = j^! \Bigg( \sum_{I \in \Ifr} \iota^I_*(\alpha_I) \Bigg) \in \Omega_\bullet(D^\circ).$$
For each $I \in \Ifr$ choose a coherent sheaf $\Fc_I$ on $D_I$ extending $E \vert_{j^{-1}(D_I)}$. As $D_I$ is regular, $\Fc_I$ is a perfect complex, and therefore has well defined Chern classes. By the naturality of Chern classes, it follows that
$$c_i(E) \bullet 1_{D^\circ / \Zb} = j^! \Bigg( \sum_{I \in \Ifr} \iota^I_* \big(c_i(\Fc_I) \bullet \alpha_I\big) \Bigg) \in \Omega_\bullet(D^\circ),$$
which is exactly what we wanted.
\end{proof}

Combining the Extension theorem with the projective bundle formula, we obtain the following $\Ab^1$-invariance statement.

\begin{cor}[Homotopy invariance]\label{HomotopyInvarianceCor}
Let $A$ be as in Theorem \ref{ExtensionThm}, let $X$ be a quasi-projective derived $A$-scheme and let $p: E \to X$ be a vector bundle of rank $r$ on $X$. Then the pullback map
$$p^!: \Omega^A_*(X)[e^{-1}] \to \Omega^A_{*+r}(E)[e^{-1}]$$
is an isomorphism.
\end{cor}
\begin{proof}
As $p$ admits a section, $p^!$ is at least an injection. On the other hand by the projective bundle formula, we have that the morphism
$$\bigoplus_{i=0}^r \Omega^A_{*-r+i} (X)[e^{-1}] \to \Omega^A_*\big(\Pb(E \oplus \Oc)\big)[e^{-1}]$$
where the $i^{th}$ morphism is defined as
$$c_1(\Oc(1))^i \bullet \bar p^!(-)$$
is an isomorphism, where $\bar p$ is the natural projection $\Pb(E \oplus \Oc) \to X$. As $\Oc(1)$ has a global section with derived vanishing locus $\Pb(E)$ whose complement in $\Pb(E \oplus \Oc)$ is the open immersion $j: E \hook \Pb(E \oplus \Oc)$, it follows that all classes of form $c_1(\Oc(1))^i \bullet \bar p^!(\alpha)$ vanish when pulled back along $j$.  As $j^!$ is surjective by the Extension Theorem, we can conclude that $p^!$ is a surjection, which finishes the proof.
\end{proof}

\bibliographystyle{alphamod}
\bibliography{references}{}

\Addresses
\end{document}